\documentclass{amsart}
\usepackage{verbatim,amsfonts,color,mathrsfs,stmaryrd,graphicx,mathdots}
\usepackage{xcolor}
\usepackage{tikz,tikz-cd,tkz-euclide}
\usepackage{amsmath,amsthm,amssymb,latexsym,setspace,graphicx,float,pgfplots,lscape,tabularx,hyperref}
\usetikzlibrary{decorations.markings,math,patterns,positioning,arrows.meta}

\theoremstyle{plain}
\newtheorem{Thm}{Theorem}
\newtheorem{Cor}[Thm]{Corollary}

\newtheorem{Lem}[Thm]{Lemma}
\newtheorem{Prop}[Thm]{Proposition}

\theoremstyle{definition}
\newtheorem{Def}[Thm]{Definition}
\newtheorem{Rmk}[Thm]{Remark}
\newtheorem{Eg}[Thm]{Example}

\theoremstyle{remark}

\begin{document}
\tikzset{->-/.style={decoration={markings,mark=at position #1 with {\arrow{>}}},postaction={decorate}}}
 
\title{Canonical translation surfaces for computing Veech groups}

\author{Brandon Edwards}
\address{Intel Corporation JF2\\
2111 NE 25th Ave.\\
Hillsboro, OR 97124}
\email{brandon.edwards@intel.com}

\author{Slade Sanderson}
\address{Oregon State University\\Corvallis, OR 97331}
\email{sandesla@oregonstate.edu}

\author{Thomas A. Schmidt}
\address{Oregon State University\\Corvallis, OR 97331}
\email{toms@math.orst.edu}
\date{23 August 2021}


\begin{abstract}  For each stratum of the space of translation surfaces, we introduce an infinite translation surface containing in an appropriate manner a copy of every translation surface of the stratum.   Given a translation surface $(X, \omega)$ in the stratum, a matrix is in its Veech group  
$\mathrm{SL}(X,\omega)$ if and only if an associated affine automorphism of the infinite surface sends  each of a finite set,   the ``marked" {\em Voronoi staples}, arising from orientation-paired segments appropriately perpendicular to Voronoi 1-cells, to another pair of orientation-paired ``marked" segments.  

We prove a result of independent interest.   For each real  $a\ge \sqrt{2}$  there is an explicit hyperbolic ball such that for any Fuchsian group trivially stabilizing $i$,   the Dirichlet domain centered at $i$ of the  group already agrees within the ball with the intersection of the hyperbolic half-planes determined by the group elements whose Frobenius norm is at most $a$.  

Together, these results give rise to a new algorithm for computing Veech groups. 
\end{abstract}

\maketitle

\tableofcontents
 
\section{Introduction}\label{s:intro}      Translation surfaces have been studied for decades,  motivated by their appearance in various areas of mathematics, as well as by their intrinsic beauty.    Although their study can be said to have begun with Teichm\"uller in the 1940s, it was work of Thurston in the 1970s and then Masur and Veech in especially the 1980s that brought this to the forefront.    Results since that time have been numerous, with many further celebrated results.     

 \subsection{Canonical surface per stratum} 
Any non-zero holomorphic 1-form $\omega$ on a closed Riemann surface $X$ gives local coordinates by integration, the result is a translation surface  $(X, \omega)$, see say \cite{Masur}.    The zeros of $\omega$ result in singularities of the ``flat" metric of the translation surface,  the translation surfaces of the same set of orders of zeros form a {\em stratum}.     Although in its natural topology, almost every stratum has more than one connected component \cite{KontsevichZorich},  we associate to each stratum in a natural manner a single (infinite and, in general, non-connected) translation surface (see Definition~\ref{d:canTranSurf})  and show that every translation surface in the stratum can be appropriately represented on this canonical translation surface.    More precisely, every translation surface of the stratum has an isometric copy of its Voronoi 2-cells on the canonical surface,   allowing reconstruction of  the original translation surface (see Proposition~\ref{p:reconstructingXomega}).

 \subsection{Veech groups}    Various dynamical properties of the linear flow on a translation surface are determined by its Veech group, \cite{Veech}.   Every Veech group is a non-cocompact Fuchsian group \cite{Veech}.     By way of Teichm\"uller theory, when the group is a lattice, there is a corresponding algebraic curve in Riemann moduli space, $\mathcal M_g$ where $g$ is the genus of $X$, which is embedded with respect to the Teichm\"uller metric \cite{Veech};  Veech named such curves {\em Teichm\"uller curves},   the translation surface is said to be a {\em lattice surface}.       Veech also showed that certain triangle groups arise as what we call Veech groups,  Bouw-M\"oller \cite{BouwM} showed that up to finite index every non-cocompact Fuchsian triangle group so arises.    On the other hand,  almost every translation surface has a trivial Veech group \cite{Moller}; and, whereas any polygon whose vertex angles are rational multiples of $\pi$ can be ``unfolded" to achieve a translation surface, only three non-isosceles acute Euclidean triangles give lattice surfaces \cite{KenyonSmillie, Puchta}. 
Although  constraints on which non-cocompact Fuchsian groups can be Veech groups,  such as equality of trace field with invariant trace field if the group contains a hyperbolic element (essentially a result of \cite{KenyonSmillie}) are known,   a complete list of Fuchsian groups realized as Veech groups remains to be determined.  

We give a new algorithm for computing Veech groups of translation surfaces (compact and without boundary),  see \S\ref{ss:algoOut}.  For a lattice surface, our algorithm completely computes the group.     Our approach is different from the general algorithms in the literature \cite{Bowman} (see also \cite{Veech2}), \cite{ BroughtonJudge, Mukamel, SmillieWeiss} (implementations exist for the algorithms of Mukamel \cite{Mukamel} and, as of  quite recently by S.~Freedman \cite{Freedman}, of Bowman \cite{Bowman}), 
as well as from the special case algorithms \cite{Schmithusen, Freidinger}.   As with most of these others, our algorithm has two parts:  determining group elements and building a fundamental domain for the action of the group on the hyperbolic plane. 

\subsection{New criterion for membership in Veech group}
  To determine elements, we view a translation surface as being assembled from the disjoint union of its Voronoi 2-cells by identifying shared edges.     There is a finite set of saddle connections whose images on the aforementioned canonical surface lead to a reconstruction of the original surface by identifying the two edges lying on the  perpendicular  bisector of the image of a saddle connection and that of its orientation-reversed saddle connection.    We call these pairs of saddle connections {\em Voronoi staples}, and  show that an element of $\mathrm{SL}_2(\mathbb R)$ is in the Veech group exactly if it sends  the Voronoi staples on the canonical surface to copies of orientation-paired saddle connections, Proposition~\ref{p:memCrit}.

\bigskip 
\subsection{New result on construction of Dirichlet domains}
 While we build a fundamental domain in a standard manner, we insist on finding elements in increasing Frobenius norm.   We show a result of independent interest:  For any Fuchsian group (trivially stabilizing $z=i$),   the intersection of the half-planes determined by the elements of any explicitly bounded Frobenius norm agrees with the Dirichlet domain based at $z=i$ of the group within a corresponding explicit hyperbolic ball.   See Proposition~\ref{p:agreeingWithinBalls}.   
 
 Recall that  the Dirichlet domain  is the nested (appropriately decreasing) limit of the convex bodies defined by intersecting the half-planes appropriately defined by finite sets of elements.   Thus, the new result is that  by taking the Frobenius norm ordering,  there are also common domains which converge (appropriately increasing) to the interior of the Dirichlet domain.    With this,  we give a test for when the elements of a bounded Frobenius norm generate the group, see Theorem~\ref{t:done!}.   This test thus gives a stopping condition for our algorithm, one that can only be fulfilled if the group is a lattice.   
 
 \subsection{Examples, and related ongoing work}
 We have confirmed the computation of various Veech groups using the algorithm \cite{Sanderson}.   Here we use a simple example to illustrate matters (see Example~\ref{eg:MustRespectOrientationPairing}, Figure~\ref{f:memberFailNoPairingRespect}, and Subsection~\ref{ss:theL}).  We also compute elements within an infinitely generated Veech group, finding that the first   tens of thousands of elements in order of  Frobenius norm are all contained in a subgroup generated by three elements (see Subsection~\ref{ss:InfGen} ).
 
In on-going work \cite{Sanderson2}, the second named author has shown how the approach of this paper can be reversed so as to begin with a Fuchsian group and determine those translation surfaces which have it as their Veech group.

 \bigskip
 This paper is based on the Ph.D. dissertation \cite{Edwards}  of the first named author, as well as the implementation by the second named author \cite{Sanderson}.  That implementation was built upon  Delecroix-Hooper's ``flat surface" computer code package  \cite{DelecroixHooper}  written in  SageMath \cite{Sage}.

\section{Background}\label{s:back}

We introduce basic notation as well as remind the reader of some standard results. 

\subsection{Translation surfaces}   See, say,  \cite{Masur, Moller, Wright} for much of the following.  

\subsubsection{Three views of a translation surface.}   A {\em translation surface} is a real surface $X$ such that on the complement of a finite set of points  $X \setminus \Sigma$ there is an atlas whose transition functions are all translations, and such that the flat structure on this complement extends to $X$ to have cone singularities of angles integral multiples of $2\pi$.    Equivalently, a  translation surface is a Riemann surface $X$ with a non-zero holomorphic differential (also called an abelian differential) $\omega$ where $\Sigma$ is the set of zeros of $\omega$; due to this view we often use the notation $(X, \omega)$ for a translation surface.    A third equivalent form is as a collection of polygons in the Euclidean plane,  with equal length parallel sides identified by translation (with the result an oriented surface); singularities occur only at points arising from vertices.   We say $(X, \omega)$  in this form is {\em polygonally presented}.    

It is common to change without warning from one to another of these perspectives when discussing any given translation surface.     We will also use without warning the fact that subsurfaces of a translation surface are naturally translation surfaces.

Note that  it is sometimes useful to allow the inclusion of removable singularities, thus points of cone angle $2 \pi$, in $\Sigma$.

\subsubsection{Affine diffeomorphisms and the Veech group}     The group of translations of the plane is $\mathrm{Trans}(\mathbb R^2) \cong \mathbb R^2$.    A {\em translation map} between translation surfaces is a continuous map sending singularities to singularities which is a translation with respect to the local coordinates given by the respective translation atlases.   
Recall that the affine group of the plane is $\mathrm{Trans}(\mathbb R^2)  \rtimes \mathrm{GL}_2(\mathbb R)$.      
An {\em affine diffeomorphism} of a translation surface is a homeomorphism $f:X \to X$ which sends $\Sigma$ to itself, while inducing an affine diffeomorphism on $X\setminus \Sigma$. That is, the maps on local coordinates (as induced in the usual manner) are elements of the affine group of the plane.   Under composition, the set of all of these forms a group, $\mathrm{Aff}(X, \omega)$.    Because of the normality of translations in the affine group of the plane, the linear part of each affine diffeomorphism $f$ is constant, independent of choice of local coordinates,  and thus there is a group homomorphism $\mathrm{der}: \mathrm{Aff}(X, \omega) \to \mathrm{GL}_2(\mathbb R)$.    The image  of the subgroup of orientation-preserving affine diffeomorphisms, $\mathrm{Aff}^+(X, \omega)$,  is called the {\em Veech group} of $(X, \omega)$.    We denote the kernel of $\mathrm{der}$ by $\mathrm{Trans}(X, \omega)$.

\bigskip
\noindent 
{\bf Convention.}    Unless otherwise stated, the notation  $(X, \omega)$ will denote a closed translation surface: compact without boundary. 

\bigskip 
The flat structure on any translation surface allows Lebesgue measure on the plane to  induce a measure on the surface.  When the translation surface is compact this is a finite measure, and it is an observation of Veech \cite{Veech}  that the Veech group is then a subgroup of $\mathrm{SL}_2(\mathbb R)$.  To emphasize this, one denotes the Veech group by $\mathrm{SL}(X, \omega)$.    Veech \cite{Veech} also showed that $\mathrm{SL}(X, \omega)$ is a discrete subgroup of   $\mathrm{SL}_2(\mathbb R)$.

\subsubsection{Saddle connections, strata, matrix group action}
 It is easy to see that the foliation of $\mathbb R^2$ by horizontal lines  induces a foliation on $X\setminus \Sigma$.   The lines of this foliation meet at points of $\Sigma$ of cone angle greater than $2\pi$ to give saddle singularities.

 Any geodesic ray emanating from a singularity is called a {\em separatrix}.   A separatrix joining two singularities (with no other singularity in its interior) is called a {\em saddle connection}, or simply as in \cite{KenyonSmillie} a  {\em segment}.
  Integration along a segment results in the corresponding element of $\mathbb C$, called a {\em holonomy vector}.    
 
 The collection of all (closed) translation surfaces $(X, \omega)$ with the holomorphic 1-form $\omega$ having zeros of orders $d_1, \dots, d_s$ (with repetition allowed) is called a   {\em stratum} and denoted $\mathcal H(d_1,\dots,  d_s)$.   Note that the cone angle at a zero of order $d$ is  $2 (d+1) \pi$.   By either the Riemann-Roch Theorem or that of Gauss-Bonnet,  all translation surfaces of a fixed stratum have the same genus.    
 
 The group $\mathrm{GL}_2(\mathbb R)$ acts on the collection of all translation surfaces by way of post-composition with local coordinate functions.   This action preserves each stratum.  In the setting of closed translation surfaces, Veech showed that the $\mathrm{SL}_2(\mathbb R)$-stabilizer of $(X, \omega)$ under this action is isomorphic to  $\mathrm{SL}(X, \omega)$.    Furthermore, for any $M \in \mathrm{SL}_2(\mathbb R)$, one has $M \mathrm{SL}(X, \omega) M^{-1} = \mathrm{SL}(M\cdot (X, \omega)\,)$.  
 
\subsubsection{Voronoi decomposition}  Each translation surface has a flat metric with conical singularities on it.     Masur-Smillie \cite{MasurSmillie} sketch the theory of Voronoi decompositions subordinate to the set of singularities.      A Voronoi 2-cell is an open, connected, set of points which are closer to some singularity than to any other.   The boundary of each Voronoi 2-cell is a union of Voronoi 1-cells consisting of open geodesics (each of whose points have exactly two distinct shortest paths to $\Sigma$), meeting in Voronoi 0-cells  (single points each having at least three distinct shortest paths to $\Sigma$).     Each compact $(X, \omega)$ is the union of finitely many open Voronoi 2-cells (exactly one per each element of $\Sigma$), finitely many Voronoi 1-cells and finitely many Voronoi 0-cells.     We will refer to the union of a Voronoi 1-cell and its two endpoint 0-cells simply as an {\em edge}.

\section{Canonical surface, Voronoi staples and a Veech group membership criterion}\label{s:modelsStaplesMembers} 

\subsection{Canonical translation surface, $\mathcal O$}   To each stratum we associate an infinite translation surface. 
 
\begin{Def}\label{d:canTranSurf}  
Given a stratum $\mathcal H(d_1,\dots,  d_s)$, for each subscript $i$ let $\mathcal O_i$  be the infinite translation surface $(\mathbb C, z^{d_i} \,dz)$ and let $\mathcal O = \mathcal O(d_1, \dots,d_{s})$ be the disjoint union of the $\mathcal O_i$.  

Partition the  index set  $\{1, \dots, s\}$ corresponding to repeated values of the $d_i$, with  partition elements of  cardinalities say $n_1, \dots, n_t$.    
\end{Def} 

\begin{Lem}\label{l:transGps}  
Each $\mathrm{Trans}(\mathcal O_i)$ is a cyclic group of order $1 + d_i$, whose generator acts as a rotation of angle $2 \pi$.   The group $\mathrm{Trans}(\mathcal O)$ is generated by the product of  permutation groups  $S_{n_1}\times \cdots \times S_{n_t}$ acting by change of index, along with the direct product of the various $\mathrm{Trans}(\mathcal O_i)$.
\end{Lem} 
\begin{proof}  Any element of $\mathrm{Trans}(\mathcal O_i)$ maps the singularity to itself and off of this induces translations in local coordinates.  These local coordinates arise from the ramified covering $\mathcal O_i \to \mathbb C$ and thus any induced translation is in fact the identity.   In other words,   $\mathrm{Trans}(\mathcal O_i)$ restricts to  $\mathcal O_i \setminus\{0\}$ to be the deck transformation group of the covering of the punctured complex plane.  Thus,  $\mathrm{Trans}(\mathcal O_i)$ is indeed a cyclic group of order $1 + d_i$, generated by  a rotation of $2\pi$.

Acting by permutation of indices, $S_{n_1}\times \cdots \times S_{n_t}$ certainly gives a subgroup of $\mathrm{Trans}(\mathcal O)$.    Furthermore, the trivial extension to $\mathcal O$ of the action of $\mathrm{Trans}(\mathcal O_i)$,   acting as the identity on all other $\mathcal O_j$,  gives an injection of $\mathrm{Trans}(\mathcal O_i)$ into $\mathrm{Trans}(\mathcal O)$.     Any affine diffeomorphism $\varphi$ of $\mathcal O$ must preserve cone angles, and hence composition with some element of $S_{n_1}\times \cdots \times S_{n_t}$ gives an affine diffeomorphism that sends each $\mathcal O_i$ to itself.  If furthermore 
$\varphi \in \mathrm{Trans}(\mathcal O)$ then so is this composition, and hence for each $i$ it restricts to give an element of $\mathrm{Trans}(\mathcal O_i)$.   Thus, the result holds.  
\end{proof}

\begin{Lem}\label{l:sl2IsInVeechGpOfCalO}  
For each $A \in \mathrm{GL}_2\mathbb{R}$ there is an affine diffeomorphism $f_A \in \mathrm{Aff}(\mathcal O)$ such that $\mathrm{der}(f_A) = A$. 
\end{Lem} 
\begin{proof}  We first briefly sketch the result in the setting of $(\mathbb C, dz)$, see Figure~\ref{f:jacobiansForC}.   
   The surface $A \cdot (\mathbb C, dz)$ is formed by  post-composing the chart maps of $(\mathbb C, dz)$ with $A$, using the standard action of $\mathrm{GL}_2\mathbb{R}$ on $\mathbb C \cong \mathbb R^2$.  The identity map  on $\mathbb C$ gives $\mathrm{Id_{\mathbb C}}: (\mathbb C, dz) \to A\cdot (\mathbb C, dz)$, of Jacobian equal to $A$.   Similarly, the map  sending points of $\mathbb C \cong \mathbb R^2$  to their corresponding image under $A$  induces   a diffeomorphism  $F_{A}:A\cdot (\mathbb C, dz)\to  (\mathbb C, dz)$  of Jacobian equal to $I_2$, the identity matrix.   Thus the composition $F_{A}\circ \mathrm{Id_{\mathbb C}}$ is an affine diffeomorphism whose Jacobian equals $A$.  
   
\begin{figure}[h]
\begin{tikzcd}[column sep=2pc,row sep= tiny]
 (\mathbb C, dz) \ar{r}{\mathrm{Id}_{\mathbb C}} & A\cdot (\mathbb C, dz)  \ar{r}{\mathrm{F}_A} &(\mathbb C, dz)\\
   \cup                                                                   &\cup                                                                  &\cup\\                       
 U \ar{dddd}{\varphi_U = \mathrm{Id}_U} \ar{r}                   &U\ar{dddd}{A \circ \varphi_U} \ar{r}                                &A(U)\ar{dddd}{\mathrm{Id}}\\
 &&&	\\
 &&&	\\
  &&&	\\
\mathbb C \supset U  \ar[dotted]{r}                       &A(U)\ar[dotted]{r}                                              &A(U)
\end{tikzcd} 
\caption{Local coordinate calculation shows that $F_{A}\circ \mathrm{Id_{\mathbb C}}$ has Jacobian $A$.   The  similarly defined map $f_A: \mathcal O \to \mathcal O$ is also of linear part $A$.}
\label{f:jacobiansForC}
\end{figure}

The above shows that whenever $\mathcal O_i$ corresponds to a removable singularity,  there is an affine diffeomorphism of it, say $f_i$, whose linear part is $A$.     Now, for any other index $i$,   we have $\mathrm{Id_{\mathcal O_i}}: \mathcal O_i \to A\cdot \mathcal O_i$ of linear part $A$.   Furthermore, similar to $F_A$ above, we have  $\widehat F_{A,i}: A\cdot \mathcal O_i \to  \mathcal O_i$ of trivial linear part.   Letting $\widehat F_{A}: A\cdot \mathcal O \to  \mathcal O$ be the map which restricts to $\widehat F_{A,i}$ on each $A\cdot \mathcal O_i$,  one finds that $f_A =  \widehat F_{A} \circ  \mathrm{Id_{\mathcal O}}$ is a  homeomorphism taking singularities to singularities and whose linear part  is $A$.  
\end{proof}

The locally length-minimizing paths on $\mathcal O$ are easily described.  
\begin{Lem}\label{l:geodesic}   For each $\mathcal O_i$, any two distinct points are connected by a unique geodesic.  Each such geodesic is either a straight line segment,  or a connected sequence of such segments meeting at the singularity with (minimum) angle at least $\pi$.  
\end{Lem} 
\begin{proof}   
The description of the geodesics is a special case of Strebel's result for an arbitrary (half-)translation surface $(Y, \alpha)$, see Theorem~8.1 of \cite{Strebel}.     Since $\mathcal O_i$ is simply connected, any two points of $\mathcal O_i$ are connected by a unique geodesic (see \cite{Strebel} Theorem~14.2.2 and Corollary 18.2).  
\end{proof}

\begin{Lem}\label{l:surfaceIntoO} 
Suppose that $(X, \omega)$ is an element of $\mathcal H(d_1, \dots, d_{s})$.    Then there is an injective translation map of the disjoint union of  the open Voronoi 2-cells of $(X, \omega)$ into $\mathcal O$.   This map is unique, up to the action of $\mathrm{Trans}(\mathcal O)$ on the image. 
\end{Lem} 
\begin{proof}  Enumerate the elements of  $\Sigma$ so that for each $i$,  $\sigma_i \in \Sigma$ has cone angle  $2 ( 1+d_i)\pi$.    Fix $i$ and for typographic ease   let    $\sigma =  \sigma_i$,   and  let  $\mathcal C_{\sigma}$ be the Voronoi 2-cell centered at $\sigma$.

   We can choose and fix a horizontal ray from the origin of $\mathcal O_i$ to measure angles from, and define a generalized system of polar coordinates in the obvious fashion.  That is,  every point other than the origin of $\mathcal O_i$ is uniquely represented by an ordered pair of positive real number and angle between $0$ and $2(d_i+1) \pi$.

We introduce a similar generalized polar coordinate system in $\mathcal C_{\sigma}$.   By definition,  the open set $\mathcal C_{\sigma}$ contains no other singularity, hence it contains  some  singular coordinate patch about $\sigma$.  By shrinking as necessary, this coordinate patch is a totally ramified covering of some disk centered at the origin of $\mathbb C$, of  degree $d_i+1$.   In this patch we can choose a lift of a ray along the positive $x$-axis and identify each point other than $\sigma$ by an ordered pair of positive real number and  angle between $0$ and $2(d_i+1) \pi$.   Since each point of $\mathcal C_{\sigma}$ is connected to $\sigma$ by a unique shortest path and these paths vary continuously in $\mathcal C_{\sigma}$,  we can extend the polar coordinates  to all of $\mathcal C_{\sigma}$.   The injective map $\iota_{\sigma}: \mathcal C_{\sigma} \to \mathcal O_i$  sending a point to the corresponding point with the same polar coordinates is then easily seen to be a translation map.

Now we let $\sigma = \sigma_i$ vary and define a map from the disjoint union of the Voronoi 2-cells to $\mathcal O$ to be given by applying the appropriate $\iota_{\sigma}$.   This is certainly an injective translation map. 
Given another such translation map, we can compose with an element of $S_{n_1}\times \cdots \times S_{n_t}$ to ensure that both send the Voronoi 2-cell of $\sigma_i$ to $\mathcal O_i$ for each $i$.      The only remaining freedom is accounted for by the action of each $\mathrm{Trans}(\mathcal O_i)$.   The result thus holds.  
\end{proof}

\subsection{Recovering $(X, \omega)$}  We give some details of the reconstruction of $(X,\omega)$ from the collection of the images on the canonical surface of the Voronoi 2-cells.\\

We fix some notation. 
\begin{Def}\label{d:basicNotationAtSings}  
We will assume throughout that some  injective translation map $\iota$ as above of the union of  the Voronoi 2-cells of $(X, \omega)$ into $\mathcal O$ has been fixed.  For each singularity $\sigma$ we let $\mathcal O_{\sigma} = \mathcal O_i$ where  $\iota$ sends $\sigma$ to the origin of $\mathcal O_i$.   As in the previous proof, we also let  $\mathcal C_{\sigma}$ be the Voronoi 2-cell centered at $\sigma$ and $\iota_{\sigma}: \mathcal C_{\sigma} \to    \mathcal O_{\sigma}$ the corresponding translation map.  Unless otherwise stated,  $0$ will denote the singularity of $\mathcal O_{\sigma}$.  \end{Def} 

\bigskip
 Recall that  $(X, \omega)$ is isometric to the quotient space of the disjoint union of the closure of its Voronoi cells  under the relation defined by the identification of shared edges.    Similarly, each $\iota(\mathcal C_{\sigma})$ has a collection of line segments comprising its closure which can be identified so as to recover $(X, \omega)$.   Note that for each $\sigma$, we have $\overline{\iota(\mathcal C_{\sigma})} = \overline{\iota_{\sigma}(\mathcal C_{\sigma})}$.

\begin{Prop}\label{p:reconstructingXomega}   Let $Y$ be the disjoint union of the various $\overline{\iota(\mathcal C_{\sigma})}$.    For each $\sigma$ any edge of $\mathcal C_{\sigma}$ can be uniquely identified with a line segment in $\overline{\iota(\mathcal C_{\sigma})}$ so that the two line segments corresponding to a  common edge of Voronoi 2-cells are translation equivalent  and such that the resulting equivalence relation on $Y$ results in a quotient metric space that is isometric to  $(X, \omega)$.   
\end{Prop} 
\begin{proof}    Fix $\sigma$ and choose any edge $e$ of $\mathcal C_{\sigma}$ as well as a point $p$  in the interior of $e$.  The polar coordinates of  $\mathcal C_{\sigma}$ extend to give local coordinates in some sufficiently small neighborhood of this regular point.    Since the inverse function $j$  
taking $\iota(\mathcal C_{\sigma})$ to $\mathcal C_{\sigma}$ is given by identifying points of the same polar coordinates, it extends linearly to include all of this neighborhood in its image.     Suppose that $e$ is shared with $\mathcal C_{\sigma'}$ and let $j'$ be the inverse function in this setting.  Polar coordinates here also give local coordinates in some sufficiently small common neighborhood $U$  of $p$.     These two local coordinates on $U \subset (X, \omega)$ differ by a translation,  from which it follows that $j^{-1}(U)$ is a translation equivalent copy of $(j')^{-1}(U)$ in $\mathcal O_{\sigma}$.

 Figure~\ref{f:edgesTranslEquiv} is related to the following construction.   Let $\hat{U}$ be the set taken by $j$ to $U$ and let $\hat{p} \in \hat{U}$ be the preimage of $p$. There is a unique shortest path from $\sigma$ to the regular point $p \in (X, \omega)$.   Correspondingly there is a   radial line from   $0$ to $\hat{p}$. As we vary $p\in U$ (while remaining in the interior of $e$), the  radial lines vary continuously.  We have the analogous situation with respect to the  shortest paths from $\sigma'$.      The translation equivalence can thus be extended to the set comprised of the radial lines in $\mathcal C_{\sigma'}$ (the lines translating `rigidly' with their endpoints).    Note that $\sigma'$ is sent to some regular point, say $w \in \mathcal O_{\sigma}$.       As well,   each of our  $\hat{p}$ has equal length straight line segments to $0$ and to $w$.     It follows that the set of these $\hat{p}$ is a line segment on the perpendicular bisector $L$ of the line segment joining $0$ and $w$.

  Varying $p$ shows that $j$ can be extended so that all of the interior of  $e$ has a well defined $j$-preimage, and that this preimage lies on $L$ and by continuity we find that $j$ sends  a line segment $\hat{e}$ on the boundary of $\overline{\iota(\mathcal C_{\sigma})}$  to $e$,  and that all of  $\hat{e}$ lies on $L$.   It also follows that $\hat{e}$ and the subset of the closure of   $\iota_{\sigma'}(\mathcal C_{\sigma'})$ corresponding to $e$ are translation equivalent.   Thus, the announced equivalence relation exists, and the quotient space of $Y$ is indeed isometric to $(X,\omega)$.  
\end{proof}  

When $\sigma$ is fixed we will continue to use $\hat{p}$ and $\hat{e}$  to denote the preimage of a point $p$ and of an edge $e$ respectively, under the extension $j$ of the inverse of $\iota_{\sigma}$.    We also introduce  another bit of notation. 

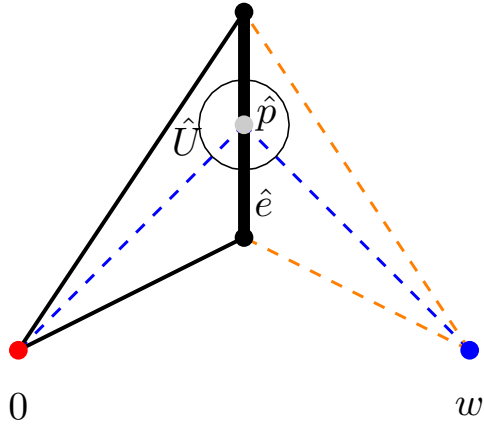
\begin{figure}
\scalebox{1.5}{
\begin{tikzpicture}
\draw [ line width=3]  (2,1) --  (2,3) node[xshift= 5, yshift = -48]{$\hat{e}$};
\draw (1.5, 1.9) node{$\hat{U}$};
\draw [ line width=0.75, blue, dashed]  (0,0) --  (2,2);
\draw [ line width=0.75, blue, dashed]  (4,0) --  (2,2);
\draw [ line width=1]  (0,0) --  (2,3);
\draw [ line width=1]  (0,0) --  (2,1);
\draw (2,2) node[scale=.5,circle,fill=black!20]{}; 
\draw (2.2,2.1)  node{$\hat{p}$};
\draw [black,domain=0:360] plot ({2 + 0.4*cos(\x)}, {2+ 0.4*sin(\x)});
\draw [orange, line width=.75, dashed]  (4,0) --  (2,3);
\draw [orange, line width=.75, dashed]  (4,0) --  (2,1);
\draw (2,1) node[scale=.5,circle,fill=black] {};
\draw (2,3) node[scale=.5,circle,fill=black] {};
\draw (0,0) node[scale=.5,circle,fill=red]{};
\draw (4,0) node[scale=.5,circle,fill=blue]{};
\draw (4, -0.5) node{$w$};
\draw (0, -0.5) node{0};
\end{tikzpicture}
}
\caption{Construction for Proposition~\ref{p:reconstructingXomega},   $\iota$-images of a common edge of two Voronoi 2-cells can be identified by a translation equivalence.    Here $\hat{p}$ is an interior point of a boundary edge $\hat{e}$ of $\iota(\mathcal C_{\sigma})$, $0$ denotes the origin of $\mathcal O_{\sigma}$,  $\hat{U}$ is a neighborhood of $\hat{p}$, $w$ is the reflection of $0$ through the line of $\hat{e}$.  Black lines outline a triangle in the closure of $\iota(\mathcal C_{\sigma})$. Blue dotted lines give the  equal length segments joining $\hat{p}$ to $0$ to $w$.} 
\label{f:edgesTranslEquiv}
\end{figure}

\begin{Def}\label{d:openDiskAboutPt}    For $z \in \mathcal O_{\sigma}\setminus \{0\}$,  let $D(z)$ be the open disk of center $z$ and radius equal to the distance from $0$ to $z$.   
\end{Def} 

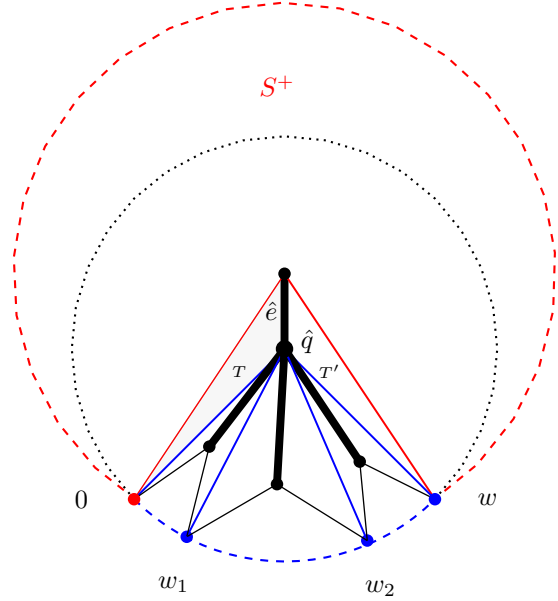
\begin{figure}
\scalebox{1}{
\begin{tikzpicture}
\draw [ line width=0, fill=black!10, opacity=0.3]  (0,0) -- (2,2) -- (2,3) -- cycle;
\draw [ line width=3]  (2,2) --  (2,3) node[xshift= -5, yshift = -14]{$\hat{e}$};
\draw [ line width=0.5, red]  (0,0) --  (2,3);
\draw (1.4, 1.7) node{\tiny{$T$}};
\draw [blue, line width=.75]  (0,0) --  (2,2);
\draw [red, line width=.75]  (4,0) --  (2,3);
\draw [blue, line width=.75]  (4,0) --  (2,2);
\draw (2.6, 1.7) node{\tiny{$T'$}};
\draw [ line width=3]  (2,2) --  (1,0.7); 
\draw (1,0.7) node[scale=.5,circle,fill=black] {};
\draw [blue, line width=.75]  (2,2) --  (0.7, -0.5) node[black, xshift= -5, yshift = -18]{$w_1$};
\draw (0.7, -0.5) node[scale=.5,circle,fill=blue] {};
\draw [black, line width=.5]  (1,0.7) --  (0.7, -0.5);
\draw [black, line width=.5]  (1,0.7) --  (0,0);
\draw [black, line width=.5]  (0.7, -0.5) --  (1.9,0.2);
\draw [ line width=3]  (2,2) --  (1.9,0.2); 
\draw (1.9,0.2) node[scale=.5,circle,fill=black] {};
\draw [ line width=3]  (2,2) --  (3,0.5);
\draw (3,0.5) node[scale=.5,circle,fill=black] {};
\draw [blue, line width=.75]  (2,2) --  (3.1, -0.55) node[black, xshift= 5, yshift = -18]{$w_2$};
\draw (3.1, -0.55) node[scale=.5,circle,fill=blue] {};
\draw [black, line width=.5]  (1.9,0.2) --  (3.1, -0.55);
\draw [black, line width=.5]  (3.1, -0.55) --  (3,0.5);
\draw [black, line width=.5]  (4, 0) --  (3,0.5);
\draw [blue,thick,dashed, domain=225:315] plot ({2 + 2*sqrt(2)*cos(\x)}, {2+ 2*sqrt(2)*sin(\x)});
\draw [thick,dotted, domain=-45:225] plot ({2 + 2*sqrt(2)*cos(\x)}, {2+ 2*sqrt(2)*sin(\x)});
\draw [red,thick,dashed, domain=-57:238] plot ({2 + sqrt(13)*cos(\x)}, {3+ sqrt(13)*sin(\x)});
\draw (2,2) node[scale=.7,circle,fill=black] {};
\draw (2,3) node[scale=.5,circle,fill=black] {};
\draw (0,0) node[scale=.5,circle,fill=red]{};
\draw (4,0) node[scale=.5,circle,fill=blue] {};
\draw (4.7, 0) node{$w$};
\draw (-0.7, 0) node{0};
\draw (2.3, 2.1) node{$\hat{q}$};
\draw [red] (1.9, 5.5) node{$S^{+}$};
\end{tikzpicture}
}
\caption{Voronoi 0-cell $q$ corresponds to an endpoint $\hat{q}$ of  a boundary line segment $\hat{e}$  of $\overline{\iota(\mathcal C_{\sigma})} \subset \mathcal O_{\sigma}$.    The proof of Proposition~\ref{p:reconstructingXomega}  shows that a reflection of the triangle $T$ through the line of $\hat{e}$ is translation equivalent to a triangle within a paired Voronoi 2-cell.  Proposition~\ref{p:goodNeighborhoods} shows that  the disk $D(\hat{q})$ is covered by translation equivalent copies of triangles in the $n+2$ (counted with possible repetition) Voronoi 2-cells on whose boundaries $q$ lies; pictured:  $n=2$.   The points $w, w_1, \dots,  w_n$ are the resulting images of the singularities.  The dashed blue arc with the black dotted arc form the boundary of $D(\hat{q})$, inside the disk is the sector $S_{-}$ (not labeled in figure).  The disk lies in the union of that sector and the sector $S^{+}$ (outer arc dotted in red) of the disk about the other end of $\hat{e}$.} 
\label{f:goodNeighborhood}
\end{figure}

\begin{Prop}\label{p:goodNeighborhoods}    Suppose that $e$ is an edge of $\mathcal C_{\sigma}$ and that $p \in e$.   Then $D(\hat{p})$  is contained in the union of  a finite number of  triangles in $\mathcal O_{\sigma}$, each of which is translation equivalent to a triangle whose vertices are the singularity and endpoints of one edge of the 2-cell of some Voronoi 2-cell of $(X, \omega)$.    In particular,  the inverse $j$ to $\iota_{\sigma}$ extends so as to take $D(\hat{p})$ to an open set of $(X, \omega)\setminus \Sigma$.
\end{Prop} 
\begin{proof}  We begin by showing  that the initial result holds for the endpoints of $e$.   Normalize  so that $\hat{e}$ is vertical and denote its  lowest point by $\hat{q}$,  see Figure~\ref{f:goodNeighborhood}.  Let $r$ be the distance from  $\hat{q}$ to $0$.   Now,  $\iota(\mathcal C_{\sigma})$ includes (the interior of) the triangle $T$ whose vertices are $0$ and the endpoints of $\hat{e}$.   The proof of Proposition~\ref{p:reconstructingXomega} shows that  the reflection of $T$ through $\hat{e}$, say  $T'$, is translation equivalent to a triangle  inside $\iota(\mathcal C_{\sigma'})$.      As in that proof, let $w$ be the reflection of $0$ through the line of $\hat{e}$.    Let $S_{-}$ be the lower sector of  $D(\hat{q})$ bounded by the radial lines ending at $0$ and $w$.     We aim to show that $S_{-}$ is contained in an appropriate union of triangles.

   Corresponding to $\hat{q}$ is an endpoint of $e$ say,   $q$.   There is a finite number of additional Voronoi 2-cells, say $\mathcal C_{\sigma_1}, \dots,  \mathcal C_{\sigma_n}$ (repetition allowed) such that $q$ is an endpoint of an edge of each;  enumerate  so that the corresponding edges are $e_1,\dots, e_{n+1}$  where   $e_1$ is shared by  $\mathcal C_{\sigma}$ and $\mathcal C_{\sigma_1}$,  $e_{n+1}$ is a shared  edge of $\mathcal C_{\sigma_n},\mathcal C_{\sigma'}$ and any other $e_i$ is shared by $\mathcal C_{\sigma_{i-1}},  \mathcal C_{\sigma_i}$.    Since  $q$ is a regular point,   the polar coordinates of $\mathcal O_{\sigma}$ allow the extension of the inverse map $j$ to include a neighborhood of $\hat{q}$.  As well, the directions of the various edges as given by geodesic segments emanating from $q$ are well defined.     This allows us to repeatedly use the construction of the proof of Proposition~\ref{p:reconstructingXomega}. Thus, for each $i$, we can translate the $\iota$-copy of a triangle  within  $\iota(\mathcal C_{\sigma_i})$ so that the vertices of the closure on $\mathcal O_{\sigma}$ are $\hat{q}$, the second endpoint  of  a line segment $\hat{e}_i$, and a point $w_i$ on the boundary arc of $S_{-}\,$.   In particular,  there is an open neighborhood of $\hat{q}$ contained in the union of all of the triangles on $\mathcal O_{\sigma}$.
   
   If $S_{-}$ is not already contained in the union of these triangles,  then    let $\hat{q}_i$ be the lower endpoints of the $\hat{e}_i$.    These also correspond to regular points,  and thus there are translates of closures of $\iota$-images of triangles which meet  at them.   Note that the reflections  involved place the images of the corresponding  singularities outside of $S_{-}\,$.   (Alternatively:  These images must be external to $S_{-}$, as if not then there is a straight path from $\hat{q}$ to such an image that gives rise to a path from $q$ to $\Sigma$ that is shorter than the radius of $S_{-}$, a contradiction.)     
    We can now continue this process, finding translated images of triangles so as to share edges, with singularities of images always external to $S_{-}\,$.   By the compactness of $(X, \omega)$ there are finitely many edges of Voronoi 2-cells, and thus the  triangles must have area bounded below by a positive constant.   Therefore,  allowing for ever more generations,  we must eventually have $S_{-}$ included in the union of a finite number of  our triangles on $\mathcal O_{\sigma}$.  
   
   We can now argue with the other endpoint of    $\hat e$ instead of $\hat q$.  Let $R$ be the distance from this point to $0$ and let $S^+$ be the sector of center this second point and radius   $R$  which is bounded by the radial lines to $0$ and to $w$.  By the arguments above $S^+$  is similarly covered by triangles.   By Euclidean geometry,  $r \le R + |\hat{e}|$  (where $|\cdot |$ denotes length).   It follows that  the disk $D(\hat{q})$ is contained in the union of $S_{-}$ and $S^+$.   By symmetry,  the disk of the other endpoint is also contained in the union of these sectors.    But, similar comparisons of lengths of paths show that $D(\hat{p}) \subset S_{-}\cup S^+$ also holds for points $\hat{p}$  lying in the interior of $\hat{e}$.      

   Finally, since none of the images of the singularities of the various triangles lies in $S_{-}\cup S^+$,   given  $\hat{p}$  the map $j$ extends so as to send $D(\hat{p})$ to an open subset of    $(X, \omega)\setminus \Sigma$.
   \end{proof}   

\bigskip

\subsection{Closed Voronoi 2-cells as convex bodies on $\mathcal O$}

We define key notions for the following.  
\begin{Def}\label{d:markedSegmentsAndOrientationPairs}    Any separatrix $\rho$ emanating from  $\sigma$  is such that the  map  $\iota$ sends $\rho \cap \mathcal C_{\sigma}$  to a line segment emanating from the origin of  $\mathcal O_{\sigma}$.   Let $\breve{\rho}$ be the open ray formed by stopping $\rho$ at the first singularity it encounters thereafter, if any such exists.    The linear continuation  of $\iota(\, \rho \cap \mathcal C_{\sigma}\,)$   whose length agrees with that of $\breve{\rho}$ then gives the unique isometric embedding of $\breve{\rho}$ into $\mathcal O$   agreeing with $\iota$ restricted to $\mathcal C_{\sigma}$.  We denote this image by $\hat{\iota}(\rho)$.  

   Given a  saddle connection $s$, we call   $\hat{\iota}(s)$ a  {\em marked segment} for $(X, \omega)$.   
\end{Def} 

\begin{Def}\label{d:halfSpace}  
Given a saddle connection $s$, let $\sigma$ be the initial endpoint of $s$.   We define  $\hat{s}$ to be the terminal endpoint of $\hat{\iota}(s)$.   We also define the {\em half-space} of $s$ to be
\[H_{\iota}(s) = \{z \in \mathcal O_{\sigma}\mid d(0, z) \le d(z, \hat{s})\},\]
where  $d(\cdot, \cdot)$ denotes the distance on $ \mathcal O_{\sigma}$ and (for simplicity's sake) $0$ denotes its origin, $\iota(\sigma)$.
\end{Def} 

\begin{Prop}\label{p:boundaryOfHalfSpace}  Given a saddle connection $s$ emanating from a singularity $\sigma$, let $s^\perp$ denote the perpendicular bisector of $\hat{\iota}(s)$.  Then $s^\perp$ is the boundary of $H_{\iota}(s)$.
\end{Prop} 
\begin{proof}  Away from its origin,  $\mathcal O_{\sigma}$ has  a Euclidean metric, and thus  $s^\perp$ is indeed well-defined.  Similarly, in the case that $\sigma$ is a removable singularity the result is well known from elementary geometry.   

In all other cases, we invoke Strebel's description of geodesics, Lemma~\ref{l:geodesic}  above.  Given two points other than the origin,    if the radial segments having these points as endpoints are of angle less than $\pi$, then the geodesic between them is simply the straight line segment connecting them in a single coordinate chart.  If their angle is greater than $\pi$,  then the geodesic is the union of two radial line segments.

Now, if $\hat{\iota}(s)$ makes an angle greater than or equal to $\pi$ with the ray from the origin to $z\in \mathcal O_{\sigma}$, then  $d(z, \hat{s}) = d(z,0) + d(0, \hat{s})$ and certainly $z \notin \partial H_{\iota}(s)$.   All other $z$ lie within an angle of $\pi$ on either side of $\hat{\iota}(s)$.  The  ramified covering $\mathcal O_{\sigma}\to \mathbb C$ sends   each sector of angle $\pi$ isometrically to $\mathbb C$,  and the result then follows from the classical case. 
\end{proof}

\begin{Def}\label{d:intersectGivesOmega}  For a singularity $\sigma$ we define the  {\em convex body} of $\sigma$ to be
\[ \Omega_{\sigma} = \cap_{s}\, H_{\iota}(s)\,,\]
where $s$ runs through the set of saddle connections emanating from $\sigma$. 
\end{Def} 

Naturally enough,  we say that a subset of any  $\mathcal O_i$ is {\em convex} if it contains the straight line segment joining any two of its points.      

\begin{Prop}\label{p:fromTheEdges}  We have equality of sets  \[  \Omega_{\sigma}  = \overline{\iota(\mathcal C_{\sigma})}.\]  
Furthermore,  to each edge $e$ of $\mathcal C_{\sigma}$ is associated a saddle connection $s$ emanating from $\sigma$ such that the image of $e$ in $\overline{\iota(\mathcal C_{\sigma})}$ lies on $s^{\perp}$.   Letting $\mathcal E_{\sigma}^{\perp}$ denote the collection of all such saddle connections, 
\[  \Omega_{\sigma}  = \cap_{s\in \mathcal E_{\sigma}^{\perp}}\, H_{\iota}(s).\]   
\end{Prop} 
\begin{proof}    We first show that $\overline{\iota(\mathcal C_{\sigma})} = \cap_{s\in \mathcal E_{\sigma}^{\perp}}\, H_{\iota}(s)$.   
Choose an edge $e$ of $ C_{\sigma}$,  and let $\sigma'$ be the singularity whose Voronoi 2-cell shares this edge.  Let $\hat{e}$ in $\overline{\iota(\mathcal C_{\sigma})}$ correspond to $e$ as above.      For $\hat{q}$ either endpoint of $\hat{e}$   by Proposition~\ref{p:goodNeighborhoods},    the polar coordinates inverse map $j$ extends to have domain including the open disk $D(\hat{q})$, which it takes to a  set of regular points of $(X, \omega)$.     With $w$ as above,   the  straight line from $0$ to $w$ has its interior contained in $D(\hat{q})$ (for at least one of the choices of $\hat{q}$) and since $j$ extends continuously to take $w$ to $\sigma'$,     there is a corresponding saddle connection $s$ on $(X, \omega)$ connecting  $\sigma$ to $\sigma'$  and such that $\hat{s} = w$.  Therefore,    $\hat{e}$ is in $s^{\perp}$.  From this it follows that $\overline{\iota(\mathcal C_{\sigma})} = \cap_{s\in \mathcal E_{\sigma}^{\perp}}\, H_{\iota}(s)$. 

  Certainly  $ \cap_{s\in \mathcal E_{\sigma}^{\perp}}\, H_{\iota}(s) \supset \Omega_{\sigma}$.   If the two convex sets were not equal then there would be some $s$ such that $s^{\perp}$ meets $\iota(\mathcal C_{\sigma})$ and in particular there must be some endpoint $\hat{q}$ of an edge $\hat{e}$ that is closer to $\hat{s}$ than to $0$.    But then $\hat{s} \in D(\hat{q})$.   From the proof of Proposition~\ref{p:goodNeighborhoods},    $D(\hat{q}) \subset S\cup S^+$  and since $0$ is on the boundary of this union,  the line segment from $0$ to $\hat{s}$ is contained in the closure of $S\cup S^+$.   Similarly to in the previous paragraph, $s$ is the image under the extension of $j$ of the line segment.    However,  this shows that the point $q$ corresponding to $\hat{q}$ is closer to $\Sigma$ than is $\hat{q}$ to $0$, a contradiction.   Thus, no such $s$ can exist. \end{proof}

\subsection{Voronoi staples determine $(X, \omega)$ }

We introduce several key sets for our considerations. 

\bigskip
 
\begin{Def}\label{d:voronoiStaples}  
Reversing orientation on any saddle connection $s$ of $(X, \omega)$ results in a saddle connection $s'$.   We say that $s, s'$ are {\em orientation-paired}, and denote the set of these pairs by $\mathcal{M}(X, \omega)$.  We then call $\hat{\iota}(s), \hat{\iota}(s')$   orientation-paired {\em marked segments}, and denote the set of  these pairs by $\widehat{\mathcal{M}}(X, \omega)$.      When $s \in \mathcal E_{\sigma}^{\perp}$,  let  $e$ denote the edge of $\mathcal C_{\sigma}$ lying on its perpendicular bisector, and $e'$ be $e$ with the opposite orientation.   Letting $\sigma'$ denote the singularity from which $s'$ emanates,  $e'$ is an edge of $\mathcal C_{\sigma'}$, from which it follows that  $s' \in \mathcal E_{\sigma'}^{\perp}\,$.    We then call $\{s, s'\}$ a  {\em Voronoi staple}, and denote the set of all Voronoi staples by $\mathcal S(X, \omega)$. The corresponding subset of $\widehat{\mathcal{M}}(X, \omega)$ is called the {\em marked Voronoi staples}, denoted  by $\widehat{\mathcal{S}}(X, \omega)$, see Figure~\ref{f:voronoiStaples}.
\end{Def}

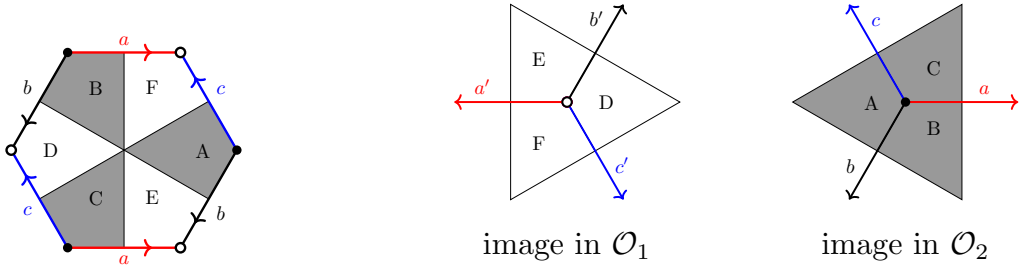
\begin{figure}[h!]
\scalebox{.75}{
\noindent
\begin{tabular}{lcr}
\begin{tikzpicture}
\draw (-0.3, 1.732) node{D};
\draw (1.5,0.9) node{E};
\draw (1.5,2.85) node{F};
\draw [fill=black!40]  (1,0) -- (1,1.732) -- (-.5,.5*1.732) -- (0,0) -- cycle node[midway, above, yshift=18]{C};
\draw [fill=black!40] (1,2*1.732) -- (1,1.732) -- (-.5,1.5*1.732) -- (0,2*1.732) -- cycle node[midway, above, yshift=-25]{B};
\draw [fill=black!40] (2.5,.5*1.732) -- (1,1.732) -- (2.5,1.5*1.732) -- (3,1.732) -- cycle node[midway, above, xshift=-10, yshift= 5]{A};
\draw [->-=.75, line width=1.2, red] (0,3.464) -- (2,3.464) node[midway, above] {$a$};  
\draw [->-=.75,line width=1.2] (3,1.732) -- (2,0) node[midway,below right] {$b$};
\draw [->-=.75,line width=1.2, blue] (0,0) -- (-1,1.732)  node[midway,below left] {$c$};
\draw [->-=.75, line width=1.2, red] (0,0) -- (2,0) node[midway,below] {$a$} node[draw,black, scale=.5,circle,fill=white] {};
\draw [->-=.75, line width=1.2] (0,3.464) -- (-1,1.732) node[midway,above left] {$b$} node[draw, black, scale=.5,circle,fill=white] {};
\draw [->-=.75,line width=1.2, blue] (3,1.732) -- (2,3.464) node[midway,above right] {$c$} node[draw, black, scale=.5,circle,fill=white] {};
\draw (0,0) node[scale=.5,circle,fill=black] {};
\draw (0,2*1.732) node[scale=.5,circle,fill=black] {};
\draw (3,1.732) node[scale=.5,circle,fill=black] {};
\end{tikzpicture}
& \phantom{tobeornottobethatis}
&
\begin{tikzpicture}   
\draw [fill=black!40] (4,1.732) -- (4,2*1.732) -- (2.5,1.5*1.732) -- (3,1.732) -- cycle node[midway, above, yshift=10]{C};
\draw [fill=black!40] (4,1.732) -- (4,0) -- (2.5,.5*1.732) -- (3,1.732) -- cycle node[midway, above, yshift=-20]{B};
\draw [fill=black!40] (2.5,.5*1.732) -- (1,1.732) -- (2.5,1.5*1.732) -- (3,1.732) -- cycle node[midway, above, xshift=-10, yshift= 5]{A};
\draw [->,line width=1] (3,1.732) -- (2,0) node[pos=0.8, above, xshift=-5] {$b$};
\draw [->, line width=1, blue] (3,1.732) -- (2,3.464) node[pos=0.7, above, xshift=5] {$c$}; 
\draw [->,line width=1, red] (3,1.732) -- (5,1.732) node[pos=0.7,  above] {$a$};
\draw (3, 1.732) node[scale=.5,circle,fill=black] {} node[below,  yshift=-60, scale = 1.6] {image in $\mathcal O_2$};
\draw (-2.3, 1.732) node{D};
\draw  [line width=0.5] (-3,1.732) -- (-2.5,1.5*1.732) -- (-1,1.732) -- (-2.5, 0.5*1.732) -- cycle;
\draw (-3.5, 1) node{F};
\draw  [line width=0.5] (-3,1.732) -- (-4,1.732) -- (-4,0)   -- (-2.5,0.5*1.732) -- cycle node[draw,scale=.5,circle,fill=white] {} node[below,  yshift=-60, scale = 1.6] {image in $\mathcal O_1$};  
\draw (-3.5,2.5) node{E};
\draw  [line width=0.5] (-4,1.732) -- (-4,2*1.732)   -- (-2.5, 1.5*1.732);  
\draw [<-,line width=1, blue] (-2,0) -- (-3,1.732) node[pos=0.2, xshift= 5,  above] {$c'$} node[draw,scale=.5,circle,fill=white] {};
\draw [<-,line width=1, red] (-5,1.732) -- (-3,1.732) node[pos=0.25, above] {$a'$}   node[draw,scale=.5,circle,fill=white] {} ;
\draw [<-,line width=1] (-2,3.464) -- (-3,1.732) node[pos=0.3, above, xshift = -5] {$b'$}  node[draw,scale=.5,circle,fill=white] {};
\end{tikzpicture}
\end{tabular}
}
\caption{Left:  Two distinct 2-cells comprising the Voronoi decomposition of the hexagonal torus subordinate to two removable singularities. Right:  The resulting image of the cells under $\iota$ as well as of the saddle connections  under $\hat{\iota}$.  The marked Voronoi staples --- $\{a, a'\}, \{b, b'\}, \{c,c'\}$ ---  show how to glue together the closure of the images of the 2-cells to obtain the surface.} 
\label{f:voronoiStaples}
\end{figure}

Our choice of terminology is to highlight the implication of  Propositions ~\ref{p:reconstructingXomega} and \ref{p:fromTheEdges}:    $(X, \omega)$  can be viewed as the disjoint union of the closures of the various $\Omega_{\sigma}$  ``stapled together" by identifying sides lying on perpendicular bisectors of members of the same marked Voronoi staple. 

\begin{Thm}\label{t:staplingTogetherVoronoi}  
Suppose that both $(X_1, \omega_1), (X_2, \omega_2)$ belong to $\mathcal H(d_1, \dots, d_{s})$ and choose maps $\iota_1, \iota_2$ to $\mathcal O$.   Then $(X_1, \omega_1), (X_2, \omega_2)$ are equivalent translation surfaces if and only if   $\widehat{\mathcal{S}}(X_1, \omega_1)$  and  $\widehat{\mathcal{S}}(X_2, \omega_2)$ are  in the same $\mathrm{Trans}(\mathcal O)$-orbit.
\end{Thm} 
\begin{proof}  First suppose that there is an element of $\mathrm{Trans}(\mathcal O)$ taking $\hat{\iota}_2(\, \mathcal S(X_2, \omega_2)\,)$ to $\hat{\iota}_1(\, \mathcal S(X_1, \omega_1)\,)$.   We can replace   $\iota_2$ by its composition with this map and assume equality of these images.   Proposition~\ref{p:fromTheEdges} then shows that for each singularity $\sigma_1$ of $(X_1, \omega_1)$ there is a singularity $\sigma_2$ of  $(X_2, \omega_2)$  so that $\Omega_{\sigma_1} = \Omega_{\sigma_2}$.  Furthermore, since each identification of corresponding pairs of sides of the various  $\Omega_{\sigma_1}, \Omega_{\sigma'_1}$ is indexed by some Voronoi pair, we find also that the $\Omega_{\sigma_2}, \Omega_{\sigma'_2}$ are identically paired.   Therefore, the resulting surfaces are isometric and in particular translation equivalent.  Proposition~\ref{p:reconstructingXomega}   then gives that  $(X_1, \omega_1), (X_2, \omega_2)$ are  translation equivalent.

If $(X_1, \omega_1), (X_2, \omega_2)$ are equivalent translation surfaces, then there is a homeomorphism $f: (X_2, \omega_2) \to (X_1, \omega_1)$  that is, off of singularities,  a translation in local coordinates.  In particular,  $f$ maps Voronoi 2-cells to Voronoi 2-cells and sends $\mathcal S(X_2, \omega_2)$ to $\mathcal S(X_1, \omega_1)$.   
 The composition $\iota_1\circ f$ then restricts to a translation map of the Voronoi 2-cells of $(X_2, \omega_2)$ into $\mathcal O$; by Lemma~\ref{l:surfaceIntoO} there is an element of $\mathrm{Trans}(\mathcal O)$ which composed with $\iota_2$ equals $\iota_1\circ f$ on the Voronoi 2-cells of $(X_2, \omega_2)$.     Equality holds then also for the extended maps, and thus  $\widehat{\mathcal{S}}(X_1, \omega_1)$ (the image of the first) is in the $\mathrm{Trans}(\mathcal O)$ -orbit of $\widehat{\mathcal{S}}(X_2, \omega_2)$.   
\end{proof}

\bigskip 
\subsection{Membership criterion}
 
 The next result applies the above with $(X_1, \omega_1) = (X, \omega)$ and $(X_2, \omega_2) = A\cdot (X, \omega)$ to give a   criterion for $A$'s membership in $\mathrm{SL}(X,\omega)$.   For this,  we use the map $f_A: \mathcal O \to \mathcal O$ of Lemma~\ref{l:sl2IsInVeechGpOfCalO}.   
\begin{Prop}[Membership Criterion, initial form]\label{p:theCriterion}\label{p:memCrit} 
Fix $(X, \omega)$ and suppose $A \in \mathrm{SL}_2\mathbb{R}$.   Then $A \in  \mathrm{SL}(X,\omega)$ if and only if   $f_A$ sends $\widehat{\mathcal S}(X, \omega)$  to a subset of $\widehat{\mathcal M}(X, \omega)$, up to some element of $\mathrm{Trans}(\mathcal O)$. 
\end{Prop} 
\begin{proof}
We first establish an intermediate result for all  $A \in \mathrm{SL}_2\mathbb{R}$, confer Figure~\ref{f:findingMarkedSegments}.    For any $A \in \mathrm{SL}_2\mathbb{R}$,  let $\hat{\iota}_A$ be the (unique up to translations) analog of $\iota$ for $A\cdot(X, \omega)$. The map $\mathrm{Id}_X: A\cdot (X, \omega) \to  (X, \omega)$ has linear part $A^{-1}$, and sends the orientation-paired saddle connections of $A\cdot (X, \omega)$ set-wise to  those of $(X, \omega)$.   Hence $\hat{\iota} \circ \mathrm{Id}_X$ sends these to $\widehat{\mathcal M}(X, \omega)$. 
Composing further with $\mathrm{Id}_{\mathcal O}: {\mathcal O} \to  A\cdot{\mathcal O}$ (whose linear part is $A$) gives that the orientation-paired saddle connections of $A\cdot (X, \omega)$ are sent isometrically in pairs into $A\cdot \mathcal O$ by $\mathrm{Id}_{\mathcal O}\circ \hat{\iota} \circ \mathrm{Id}_X$.   Thereafter composing with $\widehat{F}_A$ sends these isometrically into $\mathcal O$.   Since  $f_A = \widehat{F}_A\circ \mathrm{Id}_{\mathcal O}$,  Lemma~\ref{l:surfaceIntoO} (arguing by the density of directions of saddle connections) shows that $f_A\circ \hat{\iota} \circ \mathrm{Id}_X$ is  $\hat{\iota}_A$, up to $\mathrm{Trans}(\mathcal O)$-equivalence. Thus, we may assume equality and hence that $\hat{\iota}_A$ sends  the orientation-paired saddle connections of $A\cdot (X, \omega)$   isometrically to $ f_A(\, \widehat{\mathcal M}(X, \omega)\,)$.  
\begin{figure}[h]
\begin{tikzcd}[column sep=2pc,row sep=2pc]
\mathcal M(X, \omega)\ar{d}{\hat{\iota} }                                                       &                                                                     &                                    &                  &\mathcal M(\, A\cdot (X, \omega)\,) \ar{d}{\hat{\iota}_A} \ar[swap]{llll}{\mathrm{Id}_{X}} \\
\widehat{\mathcal M}(X, \omega) \ar[hook, dashed]{r}&\mathcal O \ar{r}{\mathrm{Id}_{\mathcal O}}\ar[bend left=60, swap]{rr}{f_A}&A\cdot \mathcal O\ar{r}{\widehat F_{A}}&\mathcal O&\widehat{\mathcal M}(\, A\cdot (X, \omega)\,)\ar[hook', dashed]{l}
\end{tikzcd} 
\caption{Up to an element of $\mathrm{Trans}(\mathcal O)$, $\hat{\iota}_A$ sends    $\mathcal M(\,A\cdot (X, \omega)\,)$  to $ f_A(\, \widehat{\mathcal M}(X, \omega)\,)$.}
\label{f:findingMarkedSegments}
\end{figure}

\medskip  
\noindent 
($\Rightarrow$) Given $A \in  \mathrm{SL}(X,\omega)$, then also $A^{-1} \in  \mathrm{SL}(X,\omega)$. By Theorem~\ref{t:staplingTogetherVoronoi},   we may assume that $\hat{\iota} (\, \mathcal S(X, \omega)\,)$ and $\hat{\iota}_{A^{-1}} (\, \mathcal S(\, A^{-1}\cdot (X, \omega)\,)\,)$  are equal.   From the previous paragraph, this is a subset of $ f_{A^{-1}}(\, \widehat{\mathcal M}(X, \omega)\,)$. It follows that $f_{A}\circ\hat{\iota} (\, \mathcal S(X, \omega)\,)$ equals some subset of $\widehat{\mathcal M}(X, \omega)$.   This first direction hence holds.  
\medskip
 
\noindent 
($\Leftarrow$) Again since $ \mathrm{SL}(X,\omega)$ is a group, it suffices to show that $A \in \mathrm{SL}(X,\omega)$ if $f_{A^{-1}}(\,\widehat{\mathcal S}(X, \omega)\,) \subset \widehat{\mathcal M}(X, \omega)$.   To show this, we apply $f_A$ and  find $\widehat{\mathcal S}(X, \omega)
 \subset f_{A}(\,\widehat{\mathcal M}(X, \omega)\,) = \widehat{\mathcal M}(\,A\cdot(X, \omega)\,)$.    Any singularity $\sigma$ of $A\cdot (X, \omega)$ is identified with a fixed singularity of $(X, \omega)$, naturally also denoted by $\sigma$.    Since  $\Omega_{\sigma}(\,A\cdot (X, \omega)\,)$ is the intersection of the half-spaces for {\em all} marked segments associated to $\sigma$,  the convex body   given by the intersection over the $s \in \mathcal E_{\sigma}^{\perp}(X, \omega)$ contains 
 $\Omega_{\sigma}(\,A\cdot (X, \omega)\,)$.  By Proposition~\ref{p:fromTheEdges}, this larger convex body is  $\Omega_{\sigma}(X, \omega)$, and is a copy of the closure of the Voronoi 2-cell for $\sigma$ on  $(X, \omega)$.

      Summing over all $\sigma$ shows  the natural Lebesgue measure of $A\cdot (X, \omega)$ is  at most the measure of $(X, \omega)$.   However,  since $A \in \mathrm{SL}_2\mathbb{R}$, the measure of $A\cdot (X, \omega)$ is the same as that of $(X, \omega)$.   Therefore, equality of convex bodies holds for each $\sigma$.      We now apply  the construction of Proposition~\ref{p:reconstructingXomega} to each of the collections of $\Omega_{\sigma}(\,A\cdot (X, \omega)\,)$ and  $\Omega_{\sigma}(X, \omega)$.   
The fact that $f_A$ respects orientation-pairings assures that the resulting two translation surfaces are equivalent.   Therefore,  $A\cdot (X, \omega)$ and $(X, \omega)$ are translation equivalent and hence $A \in \mathrm{SL}(X,\omega)$. 
\end{proof}

\begin{Rmk}\label{r:ExactStabilizer}  
In his Ph.D. dissertation \cite{Edwards}, Edwards shows that the  $\mathrm{SL}_2\mathbb{R}$ action on $\mathcal O$ induces an action on $\mathrm{Trans}(\mathcal O)$-equivalent classes of $\mathbb Z_2$-paired subsets of $\mathcal O$, and that the stabilizer of the class of $\widehat{\mathcal M}(X, \omega)$ is  exactly $\mathrm{SL}(X,\omega)$. 
 \end{Rmk} 

\bigskip
 For ease, let us say that the {\em length} of a Voronoi staple is the length of the underlying saddle connection (with one of the possible orientations).  In the following, we assume that some $\iota: (X, \omega) \to \mathcal O$ has been chosen.  Recall that the maximum singular value of a real matrix $A$  is the maximum of length of the image of unit vectors under $A$ and thus gives a bound on  how much $A$ can expand lengths.

\begin{Def}\label{d:markedPairsOfBddLength}  For each positive real number $r$, let  $\mathcal P_r =  \mathcal P_r(X,\omega)$ be the set of orientation pairs of marked segments of length at most $r$.  
\end{Def} 

\begin{Cor}[Membership Criterion, with bounds]\label{c:svpAndCatchVorStaples}  Let $\ell$ be the 
maximal length of any Voronoi staple of $(X, \omega)$.  Suppose that $\nu \in [1, \infty)$ and  $A \in \mathrm{SL}_2\mathbb{R}$ has maximum singular value at most $\nu$.

Then $A \in  \mathrm{SL}(X,\omega)$ if and only if  $f_A$  maps the set of marked Voronoi staples of $(X,\omega)$  into  some $\mathrm{Trans}(\mathcal O)$-translate of $\mathcal P_{\nu \ell}$.
\end{Cor} 
\begin{proof}[Sketch]  We use Proposition~\ref{p:theCriterion}.  
Since  $f_A$ has linear part $A$, it expands lengths by at most the maximum singular value of $A$ and hence the image of any marked Voronoi staple can have length at most $\nu \ell$.    
\end{proof}

\begin{Eg}[Respecting orientation-pairing is crucial]\label{eg:MustRespectOrientationPairing}  It is important to note that an element $A \in \mathrm{SL}_2(\mathbb R)$ may take marked segments  to marked segments, but still fail the membership criterion, due to it not  respecting  orientation-pairings.   Figure~\ref{f:memberFailNoPairingRespect} illustrates this in the case of an $(X, \omega) \in \mathcal H(2)$, formed by glueing three unit squares together in the configuration shown and identifying opposite sides,   where $A=\left(\begin{smallmatrix}1 & 0\\ 1 & 1\end{smallmatrix}\right)$ sends the horizontal marked segment $s$ to the marked segment $\gamma$,   but does not send the paired $s'$ to $\gamma'$.   Nor is there any element of $\mathrm{Trans}(\mathcal O)$ that can resolve this.    However, the reader may wish to verify that $B=\left(\begin{smallmatrix}1 & 0\\ 2 & 1\end{smallmatrix}\right)$ does fulfill the membership criterion.
\end{Eg} 

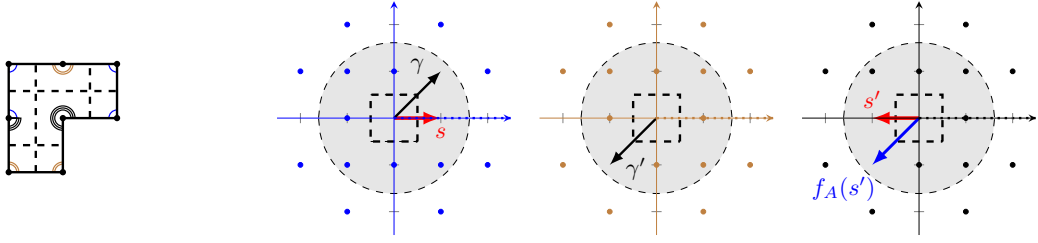
\begin{figure}[t]
\scalebox{0.9}{
\begin{tabular}{lccr}
\begin{minipage}{.25\textwidth}
\begin{tikzpicture}
\tikzmath{\x0=0; \y0=0; \scale=.8; \sectorscale=.15*\scale;}
\draw [line width=.5, blue] (\x0+\sectorscale\x0,\y0+\scale) arc (0:90:\sectorscale) -- (\x0,\y0+\scale) -- cycle;
\draw [line width=.5, blue] (\x0+2*\scale,\y0+\scale+\sectorscale\y0) arc (90:180:\sectorscale) -- (\x0+2*\scale,\y0+\scale) -- cycle;
\draw [line width=.5, blue] (\x0+2*\scale-\sectorscale\x0,\y0+2*\scale) arc (180:270:\sectorscale) -- (\x0+2*\scale,\y0+2*\scale) -- cycle;
\draw [line width=.5, blue] (\x0,\y0+2*\scale-\sectorscale\y0) arc (270:360:\sectorscale) -- (\x0,\y0+2*\scale) -- cycle;
\draw [line width=.5, brown] (\x0+\sectorscale\x0,\y0) arc (0:90:\sectorscale) -- (\x0,\y0) -- cycle;
\draw [line width=.5, brown] (\x0+1.25*\sectorscale\x0,\y0) arc (0:90:1.25*\sectorscale) -- (\x0,\y0) -- cycle;
\draw [line width=.5, brown] (\x0+\scale,\y0+\sectorscale\y0) arc (90:180:\sectorscale) -- (\x0+\scale,\y0) -- cycle;
\draw [line width=.5, brown] (\x0+\scale,\y0+1.25*\sectorscale\y0) arc (90:180:1.25*\sectorscale) -- (\x0+\scale,\y0) -- cycle;
\draw [line width=.5, brown] (\x0+\scale-\sectorscale\x0,\y0+2*\scale) arc (180:360:\sectorscale) -- cycle;
\draw [line width=.5, brown] (\x0+\scale-1.25*\sectorscale\x0,\y0+2*\scale) arc (180:360:1.25*\sectorscale) -- cycle;
\draw [line width=.5] (\x0+\scale+\sectorscale\x0,\y0+\scale) arc (0:270:\sectorscale) -- (\x0+\scale,\y0+\scale) -- cycle;
\draw [line width=.5] (\x0+\scale+1.25*\sectorscale\x0,\y0+\scale) arc (0:270:1.25*\sectorscale) -- (\x0+\scale,\y0+\scale) -- cycle;
\draw [line width=.5] (\x0+\scale+1.5*\sectorscale\x0,\y0+\scale) arc (0:270:1.5*\sectorscale) -- (\x0+\scale,\y0+\scale) -- cycle;
\draw [line width=.5] (\x0+\sectorscale\x0,\y0+\scale) arc (360:270:\sectorscale) -- (\x0,\y0+\scale) -- cycle;
\draw [line width=.5] (\x0+1.25*\sectorscale\x0,\y0+\scale) arc (360:270:1.25*\sectorscale) -- (\x0,\y0+\scale) -- cycle;
\draw [line width=.5] (\x0+1.5*\sectorscale\x0,\y0+\scale) arc (360:270:1.5*\sectorscale) -- (\x0,\y0+\scale) -- cycle;
\draw [line width=1] 
(\x0,\y0) node[scale=\scale*.35,circle,fill=black] {} -- 
(\x0+\scale,\y0) node[scale=\scale*.35,circle,fill=black] {} -- 
(\x0+\scale,\y0+\scale) node[scale=\scale*.35,circle,fill=black] {} -- 
(\x0+2*\scale,\y0+\scale) node[scale=\scale*.35,circle,fill=black] {} -- 
(\x0+2*\scale,\y0+2*\scale) node[scale=\scale*.35,circle,fill=black] {} -- 
(\x0+\scale,\y0+2*\scale) node[scale=\scale*.35,circle,fill=black] {} -- 
(\x0,\y0+2*\scale) node[scale=\scale*.35,circle,fill=black] {} -- 
(\x0,\y0+\scale) node[scale=\scale*.35,circle,fill=black] {} -- 
cycle;
\draw [line width=1, black, dashed ] (\x0+0.5*\scale,\y0) -- (\x0+0.5*\scale,\y0+2*\scale);
\draw [line width=1, black, dashed] (\x0+1.5*\scale,\y0 + \scale) -- (\x0+1.5*\scale,\y0+2*\scale);
\draw [line width=1, black, dashed] (\x0,\y0+0.5*\scale) -- (\x0+\scale,\y0+0.5*\scale);
\draw [line width=1, black, dashed] (\x0,\y0 +1.5*\scale) -- (\x0+2*\scale,\y0+1.5*\scale);
\end{tikzpicture}
\end{minipage}
&
\begin{minipage}{.2\textwidth}  
\centering
\begin{tikzpicture}
\begin{axis}
[blue, 
axis on top,
axis equal image,
axis lines=center,
grid style=dashed,
width=2\textwidth,
ymax=2.5, ymin=-2.5, xmax=2.5, xmin=-2.5, ytick={-2,-1,...,2}, xtick={-2,-1,...,2}, xticklabels={,,}, yticklabels={,,}]
\addplot [only marks,fill=blue,mark size=1, blue] coordinates {(2,1)  (1,2) (0,1) (-1,2) (-1,1) (-2,1) (-1,0) (-2,-1) (-1,-1) (-1,-2) (0,-1) (1,-2) (1,-1) (2,-1)};
\addplot [domain=-sqrt(2.6):sqrt(2.6),samples=100,dashed, black, fill=black!10] {sqrt(2.6-x^2)};
\addplot [domain=-sqrt(2.6):sqrt(2.6),samples=100,dashed, black, fill=black!10] {-sqrt(2.6-x^2)};
\addplot [domain=-0.5:0.5,samples=100, black, dashed] {0.5};
\addplot [domain=-0.5:0.5,samples=100, black, dashed] {-0.5};
\draw [line width=1, black, dashed] (axis cs: 0.5, -0.5) -- (axis cs:  0.5, 0.5)-- (axis cs:  -0.5, 0.5) -- (axis cs:  -0.5, -0.5) -- cycle;
\draw [->, >=latex,line width=1.5, red] (axis cs: 0, 0) -- (axis cs:  1, 0) node[below]{$s$};
\draw [->,>=latex, line width=1, black] (axis cs: 0, 0) -- (axis cs:  1, 1) node[midway, above, yshift=5]{$\gamma$};
\draw [line width=1, dotted, blue] (axis cs: 0.0, 0.0) -- (axis cs:  2.5, 0.0);
\end{axis}
\end{tikzpicture}
\end{minipage}
&\;\;\;\;\;
\begin{minipage}{.2\textwidth}  
\centering
\begin{tikzpicture}
\begin{axis}
[brown, axis on top,
axis equal image,
axis lines=center,
grid style=dashed,
width=2\textwidth,
ymax=2.5, ymin=-2.5, xmax=2.5, xmin=-2.5, ytick={-2,-1,...,2}, xtick={-2,-1,...,2}, xticklabels={,,}, yticklabels={,,}]
\addplot [only marks,fill=brown,mark size=1] coordinates {(1,0) (2,1) (1,1) (1,2) (0,1) (-1,2) (-1,1) (-2,1) (-1,0) (-2,-1)  (-1,-2) (0,-1) (1,-2) (1,-1) (2,-1)};
\addplot [domain=-sqrt(2.6):sqrt(2.6),samples=100,dashed, black, fill=black!10] {sqrt(2.6-x^2)};
\addplot [domain=-sqrt(2.6):sqrt(2.6),samples=100,dashed, black, fill=black!10] {-sqrt(2.6-x^2)};
\draw [line width=1, black, dashed] (axis cs: 0.5, -0.5) -- (axis cs:  0.5, 0.5)-- (axis cs:  -0.5, 0.5) -- (axis cs:  -0.5, -0.5) -- cycle;
\draw [->, >=latex,line width=1, black] (axis cs: 0, 0) -- (axis cs:  -1, -1) node[below, xshift=11, yshift=7]{$\gamma'$};
\draw [line width=1, dotted, brown] (axis cs: 0.0, 0.0) -- (axis cs:  2.5, 0.0);
\end{axis}
\end{tikzpicture}
\end{minipage}
&\;\;\;\;\; 
\begin{minipage}{.2\textwidth}   
\centering
\begin{tikzpicture}
\begin{axis}
[axis on top,
axis equal image,
axis lines=center,
grid style=dashed,
width=2\textwidth,
ymax=2.5, ymin=-2.5, xmax=2.5, xmin=-2.5, ytick={-2,-1,...,2}, xtick={-2,-1,...,2}, xticklabels={,,}, yticklabels={,,}]
\addplot [only marks,fill=black,mark size=1]       coordinates {(1,0) (2,1) (1,1) (1,2) (0,1) (-1,2) (-1,1) (-2,1)  (-2,-1)  (-1,-2) (0,-1) (1,-2) (1,-1) (2,-1)};
\addplot [domain=-sqrt(2.6):sqrt(2.6),samples=100,dashed, black, fill=black!10] {sqrt(2.6-x^2)};
\addplot [domain=-sqrt(2.6):sqrt(2.6),samples=100,dashed, black, fill=black!10] {-sqrt(2.6-x^2)};
\draw [line width=1, black, dashed] (axis cs: 0.5, -0.5) -- (axis cs:  0.5, 0.5)-- (axis cs:  -0.5, 0.5) -- (axis cs:  -0.5, -0.5) -- cycle;
\draw [->, >=latex, line width=1.5, red] (axis cs: 0, 0) -- (axis cs:  -1, 0) node[above]{$s'$};
\draw [->, >=latex,  line width=1.25, blue] (axis cs: 0, 0) -- (axis cs:  -1, -1) node[below, blue, xshift = -12, yshift=-1]{$f_A(s')$}; 
\draw [line width=1, dotted] (axis cs: 0.0, 0.0) -- (axis cs:  2.5, 0.0);
\end{axis}
\end{tikzpicture}
\end{minipage}
\end{tabular}
}  
\caption[]{Illustrating Example~\ref{eg:MustRespectOrientationPairing}.  Left:  Surface formed by identifying  opposite sides, each segment is of length one, is in $\mathcal H(2)$.   Voronoi partition indicated in black, dashed.  Angle measurement beginning at horizontal of left middle vertex (first $2 \pi$ in single blue, second in double brown, third in triple black rings).  Right:  The single connected component of $\mathcal O$, also of singularity of angle $6 \pi$, represented as three copies of $\mathbb C$ glued along the positive $x$-axis in cyclic order left to right.  Red arrows show one marked Voronoi staple $\{s, s'\}$.  Black arrows show one choice,  up to action of $\mathrm{Trans}(\mathcal O)$, for $\gamma = f_A(s)$  and its orientation paired marked segment, where  $A=\left(\begin{smallmatrix}1 & 0\\ 1 & 1\end{smallmatrix}\right)$.   Blue  arrow is $f_A(s')$.  Since $\gamma' \neq f_A(s')$ (and similarly for all $\mathrm{Trans}(\mathcal O)$-images  of $\gamma$), 
 $A$ fails the membership criterion.    Also indicated: dots give endpoints of other marked segments,    shaded disks are of radius  $\nu \ell$ with $\ell, \nu$ defined as in Corollary~\ref{c:svpAndCatchVorStaples}, of values here $\ell=1$ and $\nu = \nu(\,||A||) = \sqrt{(3+\sqrt{5})/2}$, see Definition~\ref{d:minSingValFun}.}
\label{f:memberFailNoPairingRespect}
\end{figure}

\section{Nesting hyperbolic polygons agreeing on ever larger balls} \label{s:Agree} 
The previous section presents theoretical results for determining elements of a specific discrete subgroup of $\mathrm{SL}_2\mathbb{R}$.   Here  we give results on using finite lists of elements of any given Fuchsian group  $\Gamma  \subset \mathrm{PSL}_2(\mathbb R)$ to determine aspects of a fundamental domain for the action of $\Gamma$ on the Poincar\'e upper half-plane model of  hyperbolic space, $\mathbb H$.     Recycling notation, in this section we use $H, \Omega$ to again indicate half-spaces and their intersections, but in the context of the  $\mathbb H$.   

\bigskip
\noindent
{\bf Convention.}  Throughout this section, we assume that $\Gamma$ has no non-trivial elements fixing $i \in \mathbb H$.  

\medskip
(Should in fact such elements exist, then since $\Gamma$ acts properly discontinuously,  conjugating this group by a sufficiently small translation must result in a group whose stabilizer of $i$ is trivial.   This was used to determine the translation surface underlying the computations for Figure~\ref{R_psi_hyperbolic_geometry_figure}.)  

\medskip
 The standard construction of the Dirichlet domain (centered at $z= i$) for a group $\Gamma$ respecting our convention is algorithmic: one takes the intersection of half-planes one at a time, see \cite{Katok} for a textbook discussion.     We thus contemplate taking the intersection of such half-planes for all elements of Frobenius norm of at most a given bound.  As the bound goes to infinity, these convex bodies nest down to the Dirichlet domain for $\Gamma$.

However, more is true.    We show that as the Frobenius norm increases there are ever larger hyperbolic balls within which these convex bodies already match the Dirichlet domain of the full group.    This allows us to use an area computation to detect if $\Gamma$ is generated by the elements whose norm is bounded by any explicit constant.    
 
\subsection{Bounded Frobenius norm elements give Dirichlet domain within bounded ball}
 
Let  $d_{\mathbb H}(z, w)$ denote the standard hyperbolic distance between two points in the upper half-plane.   For $A \in \mathrm{SL}_2\mathbb{R}$ let $||A||$ denote the Frobenius norm of $A$, that is $||A|| =\sqrt{ \mathrm{tr}(A^tA)}$.    It is well-known that the Frobenius norm of $A$ equals the square-root of the sum of the squares of the singular values of $A$.   Note also that the Frobenius norm  descends to give a well-defined function on   $\mathrm{PSL}_2\mathbb{R}$, which we also call the Frobenius norm.
 
\begin{Def}\label{d:minSingValFun}  
Define the function 
\[
\begin{aligned}  \nu: [\sqrt{2}, \infty) &\to    [1, \infty)\\
                                           a             &\mapsto \sqrt{\dfrac{1}{2} (a^2 + \sqrt{a^4-4}\,)}\,.
\end{aligned}
\]    
Note that $\nu(||A||)$ gives the maximum singular value of $A$.                                          
\end{Def}

     The next result is presumably well-known but we do not know of a source where it is explicitly stated.
  
\begin{Lem}\label{l:distanceToBdryAndSingValue}  Suppose that $A \in \mathrm{SL}_2\mathbb{R}\setminus  \mathrm{SO}_2\mathbb{R}$.    Then the hyperbolic distance from $i$ to $A\cdot i$ is given by 
\[d_{\mathbb H}(i, A \cdot i) = 2 \log \nu(\, ||A||\,).\]

In particular,   the hyperbolic ball centered at $i$ and of radius $r=  \log \nu(\, ||A||\,)$ is contained in the half-plane of elements that are nearer to $i$ than to $A\cdot i$, 
\[ H_i(A) \supset B(i, \log \nu(\, ||A||\,)).\] 

\end{Lem} 

\begin{Def}\label{d:setGpElemsBddFrobNorm}  
For any positive real number $a$, let $\Gamma^a$ denote the set of all elements of $\Gamma$ whose Frobenius norm is at most $a$.      Furthermore, for any subset $S \subset \Gamma$, 
let $\Omega(\,S\,) = \cap_{A\in S}\, H_i(A)$.  Thus when $S$ is a subgroup of $\Gamma$,  we have that $\Omega(\,S\,)$ is its  
Dirichlet domain centered at $z=i$.
\end{Def}

  The very construction of Dirichlet domains shows that as $a$ increases,  the $\Omega(\,\Gamma^a\,)$ nest down to  $\Omega(\,\Gamma\,)$.  Since $a \mapsto    \log \nu(a)$ defines an increasing function,  the following shows that within ever growing balls the 
$\Omega(\,\Gamma^a \,)$ fully describe $\Omega(\,\Gamma\,)$. (Of course, neither must be completely contained in any such fixed ball.)
\begin{Prop}\label{p:agreeingWithinBalls}  For each $a \ge \sqrt{2}$,
\[ \Omega(\,\Gamma\,) \cap   B(i,  \log \nu(a)) =  \Omega(\,\Gamma^a\,)  \cap   B(i, \log  \nu(a)\,).\]
\end{Prop} 
\begin{proof} Since $a \mapsto \log \nu( a )$ is a strictly increasing function, Lemma~\ref{l:distanceToBdryAndSingValue} implies that also if $a <||A||$ then $H_i(A) $ contains the ball $B(i, \log \nu(\, a\,))$.   Since $\Omega(\,\Gamma\,)$ is formed from  $\Omega(\,\Gamma^a\,)$ by intersection with all such   $H_i(A)$, with $A\in \Gamma$, the result holds.
\end{proof}

\begin{figure}[t]
\scalebox{.66}{
\includegraphics{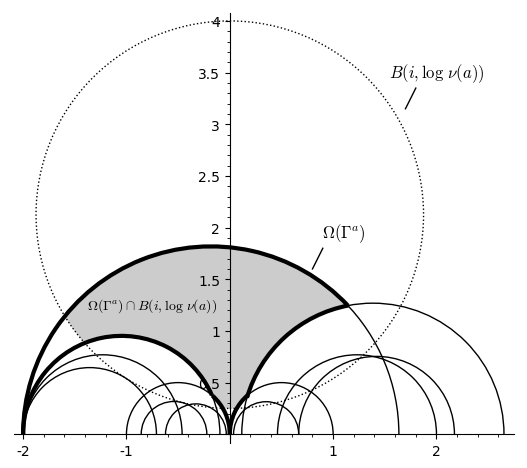}
}
\caption[a]{The shaded region shows where the computed intersection of half-planes from elements of small Frobenius norm must agree with the Dirichlet region of the full group, as guaranteed by Proposition~\ref{p:agreeingWithinBalls}.    Here,   $\Gamma =  \mathrm{SL}(M\cdot (X,\omega))$, where   $(X, \omega)$ is given in  Example ~\ref{eg:MustRespectOrientationPairing}, see also   Figure~\ref{f:memberFailNoPairingRespect}, and $M=\left(\begin{smallmatrix} 1 & 0\\ 1/2 & 1\end{smallmatrix}\right)$ is used to obtain a surface whose Veech group has trivial stabilizer of $z=i$.    Here $a = \sqrt{257}/4$.}
\label{R_psi_hyperbolic_geometry_figure}
\end{figure}

\bigskip
\subsection{Recognizing a generating set}
This subsection gives a theoretical result that leads to  a stopping condition for our algorithm to compute  $\mathrm{SL}(X, \omega)$, when this group is a lattice. 

\begin{Thm}\label{t:done!}
Suppose that $a\ge\sqrt{2}$   is such that 
\[ \mu_{_\mathbb H}(\, \Omega(\Gamma^a\,)\,) < 2\, \mu_{_\mathbb H}(\,  \Omega(\,\Gamma^a\,)  \cap   B(i,  \log  \nu(a)\,)   \,).\]
Then the subset $\Gamma^a$ generates  $\Gamma$.    In particular,  $\Gamma$ is a lattice.
\end{Thm}
\begin{proof}  Since    $\Gamma$ is a discrete group,  the subset $\Gamma^a$ is finite and hence $ \Omega(\,\Gamma^a\,)$ has finitely many sides.  By construction there is a finite  subset of $\Gamma^a$  pairing these sides. 
Let $G$ be the subgroup of $\Gamma$ generated by these side parings.  Then 
 
\[\mu_{_\mathbb H}(\Omega(G))\le \mu_{_\mathbb H}(\Omega(\Gamma^a)).\]

Proposition~\ref{p:agreeingWithinBalls} combines with this to give
\begin{align*}
\mu_{_\mathbb H}(\Omega(\Gamma))&\ge \mu_{_\mathbb H}(\,\Omega(\,\Gamma\,) \cap   B(i, \log \nu(a))\,)\\
&=\mu_{_\mathbb H}(\,\Omega(\,\Gamma^a\,) \cap   B(i,  \log \nu(a))\,)\\
&>\frac12 \mu_{_\mathbb H}(\Omega(\,\Gamma^a)\,)\\
&\ge\frac12\mu_{_\mathbb H}(\,\Omega(G)\,).
\end{align*}
Hence, the index of $G$ in the group $\Gamma$ is 
\[ |\Gamma:G|=\frac{\mu_{_\mathbb H}(\Omega(G))}{\mu_{_\mathbb H}(\Omega(\Gamma))}<2.\]
Therefore, $G=\Gamma$.  It follows that $\Gamma^a$ generates  $\Gamma$.

The initial hypothesis implies that $\mu_{_\mathbb H}(\, \Omega(\,\Gamma^a\,)\,)$ is finite.  We deduce that $\Gamma$ is finitely generated and has finite co-volume.  It is a lattice. 
\end{proof}
 
\section{From theory to algorithm}   

We sketch how the results above can be combined to give an algorithm to compute Veech groups.

\subsection{Basic steps of algorithm}\label{ss:algoOut} Given a polygonal presentation of $(X, \omega)$,  one can calculate to find a Voronoi decomposition.    Further calculation leads to the Voronoi staples, as one can explicitly calculate all saddle connections up to any given length.  Thus, the set $\widehat{\mathcal S}(X, \omega)$ and the value $\ell$ of  Corollary~\ref{c:svpAndCatchVorStaples} are knowable with finite calculation.    

   We test membership using Corollary~\ref{c:svpAndCatchVorStaples}.  We can form candidate matrices by choosing  of the saddle connections of two distinct saddle connections and a choice of image vectors.    Beginning with $a = \sqrt{2}$,  we choose image vectors from the holonomy vectors corresponding to elements of $\mathcal P_{\nu(a)\, \ell}$.   We then check whether the candidate matrix $A$ is such that $f_A$ does send  the marked Voronoi staples pairwise into orientation-paired marked segments, up to an element of  $\mathrm{Trans}(\mathcal O)$.   Exactly when this is the case is $A \in \mathrm{SL}(X,\omega)$.      We complete this round for all possible candidate matrices for this value of $a$.    We then form  $\Omega(\Gamma^a)$ and apply the obvious test implied by Theorem~\ref{t:done!}.   If this test is positive, then we have determined all of what is rightly denoted as  $\mathrm{PSL}(X, \omega)$.    (The only other way that our algorithm halts is by a user-entered upper bound on the Frobenius norm of elements to test.)    Otherwise, we increase $a$ (say by doubling its value).

\subsection{Some technical details}   We roughly indicate how some of the technical matters of implementing the algorithm are addressed.   
\begin{itemize}

\item{\bf From $\mathrm{PSL}(X, \omega)$ to $\mathrm{SL}(X, \omega)$.}  Passing from $\mathrm{PSL}(X, \omega)$ to  $\mathrm{SL}(X, \omega)$  is  merely a matter of determining if $-I_2 \in \mathrm{SL}(X, \omega)$.  Since $|| - I_2 || = \sqrt{2}$,  this is directly tested using Corollary~\ref{c:svpAndCatchVorStaples}.  
 \bigskip 

\item{\bf Finding a Veech group that trivially stabilizes $z=i$.}  The stabilizer inside of all of $\mathrm{SL}_2(\mathbb R)$ of $z=i$ is $\mathrm{SO}_2(\mathbb R)$, from which 
 a Fuchsian group $\Gamma$ follows our convention exactly if  $\Gamma^{\sqrt{2}}\cap  \mathrm{SO}_2(\mathbb R) \subset \{\pm I_2\}$.      Given a translation surface $(X, \omega)$,  we test for this by calculating $\mathrm{PSL}(X, \omega)^{\sqrt{2}}$. Upon failure,  we rather find an $M$ such that $M\cdot(X, \omega)$ trivially stabilizes $z=i$ and compute  $\mathrm{SL} (M\cdot(X, \omega)\,)$ with the algorithm as already sketched; thereafter, conjugating by $M^{-1}$ gives $\mathrm{SL}(X, \omega)$.  (Figure~\ref{R_psi_hyperbolic_geometry_figure} arises from a computation in such a setting.) It is the fact that $\mathrm{SL}(X, \omega)$ is discrete and hence that $\mathrm{PSL}(X, \omega)$ acts properly discontinuously on $\mathbb H$ (in particular, fixed points are isolated),  that implies the existence of some such $M$.   In fact,  there is some $n \in \mathbb N$ such that  $M = \left(\begin{smallmatrix} 1&0\\1/n&1\end{smallmatrix}\right)$ suffices.  (An iterative loop can determine the smallest value of $n$ for which this holds.)

 \bigskip
\item{\bf  ``Marked" data from $(X, \omega)$.}   Values of the map $\hat{\iota}$ to $\mathcal O$ are represented in terms of data structures reflecting the geometry of $(X, \omega)$:   We first choose an enumeration of the singularities;  as illustrated in Figure~\ref{f:memberFailNoPairingRespect}, for each singularity we choose one of the horizontal rays emanating from it, and measure angles beginning at that marked horizontal.  (The difference of choices of the horizontal  is accounted for by $\mathrm{Trans}(\mathcal O_i)$ as in Lemma~\ref{l:transGps}.)    Thereafter,  each saddle connection is recorded as a vector in the plane, along with  a record of the singularity from which it emanates,  as well as of which multiple of $2 \pi$ gives the copy of $\mathbb C$ in which its image lies (compare Figure~\ref{f:memberFailNoPairingRespect}),  and finally  the record includes a pointer to  its orientation-paired saddle connection.   
 
\medskip     

\item {\bf $\mathrm{SL}_2(\mathbb R)$ action on marked segments.}      When applying the membership criterion, it is crucial that one notes that for $A \in \mathrm{SL}_2(\mathbb R)$, $f_A$  can send a marked segment from one copy of $\mathbb C$ into one of the neighboring $\mathbb C$.   This is a matter of the vector being ``pushed" across the positive $x$-axis.  As an example, the reader can verify that in the setting of Example~\ref{eg:MustRespectOrientationPairing} and Figure~\ref{f:memberFailNoPairingRespect}  $f_A$ sends $1-i/2$ in the first copy of $\mathbb C$ into the second copy of $\mathbb C$. 
\end{itemize}

\section{Verified calculations}  

We have verified the results of running code based upon this algorithm for a collection of surfaces.   These include various ``L" surfaces,  particular (double) regular polygons, the famed ``eierliegende Wollmilchsau"  surface, and others, see \cite{Sanderson} for most of these.  

\subsection{An ``L"-surface}\label{ss:theL}  Here we content ourselves with a brief report on the calculation of the main running example of this paper.    For the translation surface $(X, \omega)$ of Example~\ref{eg:MustRespectOrientationPairing} and Figure~\ref{f:memberFailNoPairingRespect},  applying the algorithm, we found that $\mathrm{SL}(X, \omega)$ is generated by the matrices
\[  S =  \begin{pmatrix}
0 & -1 \\
1 & 0
\end{pmatrix},\,
A =  \begin{pmatrix}
2 & 1 \\
-1 & 0
\end{pmatrix}.\]

Note that 
$S$  squares to give $- I_2$ and that $S A$ is equal to $B =  \big( \begin{smallmatrix}
1 & 0 \\
2 & 1
\end{smallmatrix}\big)$ (confer the final sentence of Example~\ref{eg:MustRespectOrientationPairing} ).   We can easily recognize that the  ideal triangle of  vertices at $-1, 0, \infty$ is a fundamental domain for this group:   $S$ sends the $y$-axis to itself (fixing $z=i$) and   $A$ sends the geodesic of ``feet"  $-1$ and $0$ to the vertical line through $x= -1$.   Thus $\mathrm{PSL}(X, \omega)$ is a Fuchsian triangle group, of signature $(0; 2, \infty, \infty)$  --- it is of genus zero, has two cusps and one orbifold point of index $2$.  The fundamental domain has area $\pi$, and $\mathrm{SL}(X, \omega)$ is certainly a subgroup of $\mathrm{SL}_2(\mathbb R)$; our calculation thus accords with the first entry in the table of results in Section 4.2 of \cite{Schmithusen} (the ``L"-surface there is translation equivalent to our $(X, \omega)$\,).    

Note that  with $M$ as in the caption of Figure~\ref{R_psi_hyperbolic_geometry_figure},  $M^{-1}$ sends the Dirichlet domain centered at $z=i$ for $\mathrm{PSL}(M \cdot(X, \omega)\,)$ suggested in that figure to  the Dirichlet domain at $z= M^{-1}\cdot i$ for $\mathrm{PSL}(X, \omega)$.    This image is a pentagon with ideal vertices at $-1, 0$; the sides meeting at $-1$ are paired by  $A$, those meeting at $0$ by  $B$, the remaining side passes through $z=i$ and is paired with itself by $S$.   

\begin{figure}
\scalebox{.5}{
\includegraphics{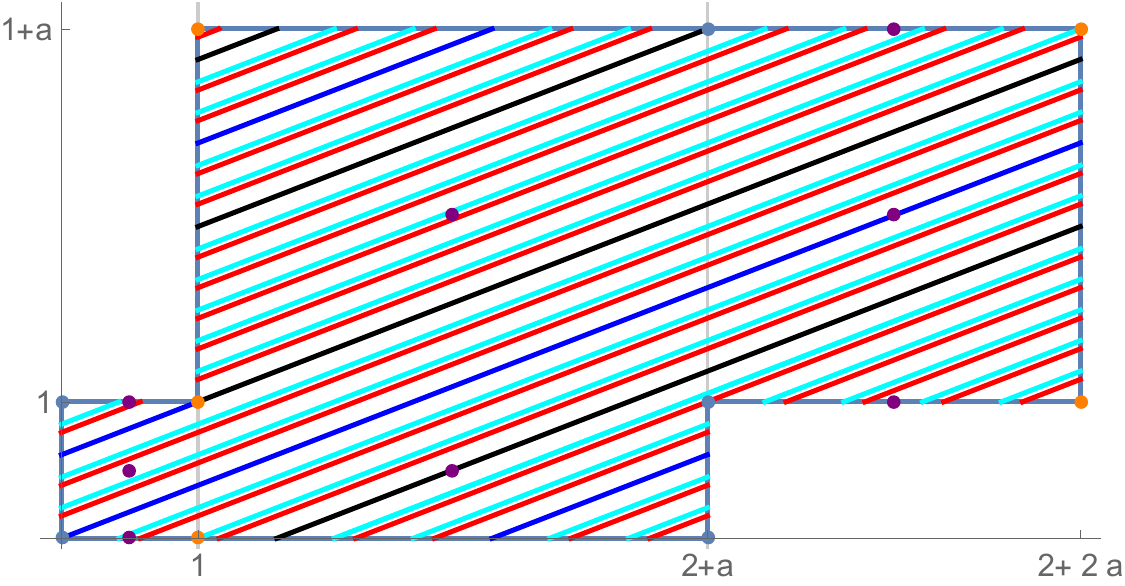}
}
\caption{Cylinder composition in a parabolic direction not in $\langle P_1, P_2 \rangle$ for   McMullen's $a=1+\sqrt{3}$ example of infinitely generated Veech group.  Here $P_1, P_2$ are parabolic elements corresponding to the vertical and horizontal directions.  The surface, formed by identifying opposite facing sides, is in $\mathcal H(1,1)$.  Marked points are the vertices, which give rise to the cone singularities, and the Weierstrass points, see \cite{McMullen}. }
\label{notInSmallSubgroup_figure}
\end{figure}
 
\subsection{Investigating an infinitely generated Veech group}\label{ss:InfGen}    McMullen \cite{McMullen} gave an infinite family of translation surfaces in $\mathcal H(1,1)$, each of whose Veech group is infinitely generated (an ``infinitely generated" group is one which cannot be finitely generated).    Our algorithm can of course find elements in order of Frobenius norm for such groups;  we investigated $\mathrm{SL}(X, \omega)$ for his surface corresponding to $a = 1 + \sqrt{3}$, see Figure~\ref{notInSmallSubgroup_figure} (and the figures in \cite{McMullen}).    Each of the vertical and the horizontal directions of this surface determine respective parabolic elements  $P_1 = \begin{pmatrix}1&0\\1&1  \end{pmatrix}, P_2 = \begin{pmatrix}1&2 + a\\0&1  \end{pmatrix} \in \mathrm{SL}(X, \omega)$.   Intriguingly enough,  we found that all of the $25,010$ elements of $\mathrm{SL}(X, \omega)$ up to Frobenius norm 256 lie in  the subgroup generated by $-I_2$ and the parabolic elements $P_1$ and $P_2$.        

This leads us to ask for the study of the  general  growth behavior in terms of $x$ of the minimal number of elements within an infinitely generated group to generate the subgroup containing all elements of Frobenius norm up to $x$.

In the meantime, we give an explicit  element of $\mathrm{SL}(X, \omega) \setminus \langle -I_2, P_1, P_2\rangle$. The Dirichlet domain of the Fuchsian $\langle P_1, P_2\rangle$ based at $z=i$ is easily seen to have vertical sides at $x = \pm (2+a)/2$ as well as  sides whose feet at infinity are $-2, 0$ and $0, 2$ and free sides $(-(2+a)/2, -2)$ and  $(2, (2+a)/2)$.    By  Theorem 9.2 of \cite{McMullen}, any direction of a separatrix containing a Weierstrass point corresponds to a parabolic element of $\mathrm{SL}(X, \omega)$.    Now,  any such direction of slope whose inverse lies in a free side of a fundamental domain cannot be the fixed point of any element of $\langle P_1, P_2\rangle$.   In particular,  using the polygonal presentation, again  see Figure~\ref{notInSmallSubgroup_figure}, the  visible geodesic path from $(0,0)$  to the Weierstrass point at  $(2+3 a/2, 1 + a/2)$ that corresponds to a parabolic fixed point of $\mathrm{PSL}(X, \omega)$ but not of $ \langle P_1, P_2\rangle$.      Let $Q \in \mathrm{PSL}(X, \omega)$ be the corresponding element.   Then the cylinder decomposition in the direction leads to $Q = \begin{pmatrix} -254 - 147 \sqrt{3}& 657 + 379 \sqrt{3}\\-99 - 57 \sqrt{3} & 
  256 + 147 \sqrt{3}\end{pmatrix}$.    Conjugation by $P_{2}^{-1}$ then gives the element in $\mathrm{SL}(X, \omega) \setminus \langle -I_2, P_1, P_2\rangle$ of smallest norm that we have found to date.


\begin{thebibliography}{9}
 
\bibitem[BM]{BouwM}  I. Bouw and M. M{\"o}ller, {\em Teichm\"uller curves, triangle groups, and Lyapunov exponents}, Ann. of Math. (2) 172, (2010), 139--185.

\bibitem[Bow]{Bowman}
J. Bowman.
\newblock{\em Teichm\"uller geodesics, Delaunay triangulations, and Veech groups.}
\newblock Teichm\"uller Theory and Moduli Problems, Ramanujan Math. Society Lecture Notes Series, Vol. 10 (2010), 113--129.

\bibitem[BrJ]{BroughtonJudge}
S. A. Broughton and C. Judge.
\newblock{\em Ellipses in translation surfaces.}
\newblock Geometriae Dedicata, Vol. 157, No. 1 (2012), 111--151.


\bibitem[DH]{DelecroixHooper}
V.~Delecroix and W.~P.~Hooper.
\newblock{\em Flat surfaces in SageMath [Computer Software].}
\newblock \url{https://github.com/videlec/sage-flatsurf}, Accessed: May 19, 2020.

\bibitem[E]{Edwards}
B. Edwards.
\newblock{\em A new algorithm for computing the Veech group of a translation surface.}
\newblock Oregon State University, PhD Dissertation (2017).

\bibitem[F]{Freedman}
S. Freedman.
\newblock{\em Private communication.}
\newblock  


\bibitem[Fr]{Freidinger}
M. Freidinger.
\newblock{\em Stabilisatorgruppen in $Aut(F_z)$ und Veechgruppen von \"Uberlagerungen.}
\newblock Diplome Thesis, Universit\"at Karlsruhe (2008). 


\bibitem[K]{Katok}
S. Katok.
\newblock{\em Fuchsian groups.}
\newblock University of Chicago Press (1992).

\bibitem[KS]{KenyonSmillie} R. Kenyon and J. Smillie, {\em Billiards in
rational-angled triangles}, Comment. Mathem. Helv. 75 (2000), 65 -- 
108.

\bibitem[KZ]{KontsevichZorich} M. Kontsevich and A. Zorich,
{\em Connected components of the moduli spaces of Abelian differentials with prescribed singularities},
\newblock Inventiones mathematicae,  153 (2003), 631--678.



\bibitem[M]{Masur}
H. Masur.
\newblock{\em  Ergodic theory of translation surfaces,}
\newblock in: Handbook of Dynamical Systems, Vol. 1 , 527--547.
Elsevier (2006).

\bibitem[MS]{MasurSmillie}  H. Masur and J. Smillie.
\newblock{\em  Dimension of Sets of Nonergodic
Measured Foliations}. Annals of Math., Second Series, Vol. 134,
No. 3 (Nov. 1991) 455--543.

\bibitem[Mc]{McMullen} C. McMullen
\newblock{\em   Teichm\"uller geodesics of infinite
complexity}.  Acta Math. 191 (2003), no. 2, 191--223. 


 
\bibitem[M\"o]{Moller}
M. M{\"o}ller.
\newblock{\em Affine groups of flat surfaces,} 
\newblock in: Handbook of Teichm\"uller
Theory Vol. 2, 369--387.
European Mathematical Society  (2009).

\bibitem[Mu]{Mukamel}
R. Mukamel.
\newblock{\em Fundamental domains and generators for lattice Veech groups.}
\newblock Comment. Math. Helv.  92 (2017), 57--83. 

\bibitem[P]{Puchta} 
J.-C. Puchta. 
\newblock{\em On triangular billiards}, 
\newblock Comment. Math. Helv. 76 (2001), no. 3, 501--505.

\bibitem[S]{Sage}
The Sage Developers.
\newblock{\em SageMath, the Sage Mathematics Software System (Version 8.7).}
\newblock 2019, \url{https://www.sagemath.org}.

\bibitem[Sa]{Sanderson}
S. Sanderson.
\newblock{\em Implementing Edwards's Algorithm for
Computing the Veech Group of a Translation
Surface.}
\newblock Oregon State University, Master's Paper (2020).

\bibitem[Sa2]{Sanderson2}
\bysame
\newblock{\em Finiteness results for Veech group realization.}
\newblock  In progress.


\bibitem[Sch]{Schmithusen}
G. Schmith\"usen.
\newblock{\em An algorithm for finding the Veech group of an origami.}
\newblock Experimental Mathematics, Vol. 13 (2004), 459--472.

\bibitem[SmW]{SmillieWeiss}
J. Smillie and B. Weiss.
\newblock{\em Characterizations of lattice surfaces.}
\newblock Inventiones mathematicae, Vol. 180 (2010), 535--557.

\bibitem[Str]{Strebel}
K. Strebel
\newblock{\em Quadratic Differentials}
\newblock Spring-Verlag, 1984.

\bibitem[V]{Veech}
W. Veech.
\newblock{\em Teichm\"uller curves in moduli space, Eisenstein series and an application to triangular billiards.}
\newblock Inventiones mathematicae, Vol. 97 (1989), 553--583. 

\bibitem[V2]{Veech2}
\bysame
\newblock{\em Bicuspid F-structures and Hecke groups.}
\newblock Proc. London Math. Soc. (3) 103 (2011) 710--745.


\bibitem[W]{Wright}
A. Wright.
\newblock{\em Translation surfaces and their orbit closures: An introduction for a broad audience.}
\newblock EMS Surveys in Mathematical Sciences 2, No. 1 (2015), 63--108.


\end{thebibliography}
\end{document}